\documentclass[12pt]{amsart}

\usepackage{times,amsfonts,amsmath,amstext,amsbsy,amssymb,
amsopn,amsthm,upref,eucal, mathrsfs}

\usepackage{url}

\def\thispartnum{I}
\input extref-arxiv.tex

\newcommand \rrad {\mathfrak{r}}

\newcommand \la {\lambda}

\DeclareMathOperator{\Span}{span}

\DeclareMathOperator{\tr}{tr}

\newcommand \Mat {{\mathrm{Mat}}}
\newcommand \Matreg{{\mathrm{Mat}_{\mathrm{reg}}}}

\newcommand \Prob {{\mathbb P}}

\DeclareMathOperator{\Leb}{Leb}

\newcommand  \const {\mathrm{const}}
\newcommand  \loc {{\mathrm{loc}}}

\newcommand  \fin {{\mathrm{fin}}}
\newcommand  \Mfin {\mathfrak{M}_\fin}

\newcommand \Conf {{\mathrm {Conf}}}

\newcommand \conf {{\mathrm {conf}}}

\newcommand\scrI{{\mathscr I}}

\newtheorem{theorem}{Theorem}[section]

\newtheorem*{teorema}{Theorem}

\newtheorem{lemma}[theorem]{Lemma}
\newtheorem{assumption}{Assumption}
\newtheorem{corollary}[theorem]{Corollary}
\newtheorem{proposition}[theorem]{Proposition}

\begin{document}
\title[The Ergodic Decomposition of Infinite Pickrell Measures. I]{Infinite Determinantal Measures and
The Ergodic Decomposition of Infinite Pickrell Measures I. Construction of infinite determinantal
measures}

\author{Alexander I. Bufetov}

\date{}
\address{ Aix-Marseille Universit{\'e}, Centrale Marseille, CNRS,   I2M}
\address{UMR 7373, rue F. Joliot Curie , Marseille, France}
\address{Steklov Institute of Mathematics,
Moscow, Russia}
\address{Institute for Information Transmission Problems,
 Moscow, Russia}
\address{National Research University Higher School of Economics,
 Moscow, }
 \address{Russia}

\begin{abstract}This paper is the first in a series of three. The main result, Theorem 1.11,  gives an explicit description of the ergodic decomposition for infinite Pickrell measures on spaces of infinite complex matrices. The main construction is that of sigma-finite analogues of determinantal measures on spaces of configurations. An example is the infinite Bessel point process, the scaling limit of sigma-finite analogues of Jacobi orthogonal polynomial ensembles. The statement of Theorem 1.11 is that the infinite Bessel point process (subject to an appropriate change of variables) is precisely the ergodic decomposition measure for infinite Pickrell measures. 
\end{abstract}

\maketitle
\tableofcontents
\section{Introduction}

\subsection{Informal outline of the main results}
The Pickrell family of measures is given by the formula
\begin{equation*}
{\mu}_n^{(s)}=\const_{n,s}\det(1+{z}^*{z})^{-2n-s}dz.
\end{equation*}
Here $n$ is a natural number, $s$  a real number, $z$ a square $n\times n$ matrix   with complex entries,
$dz$ the Lebesgue measure on the space of such matrices, and $\const_{n,s}$ a normalization constant whose precise choice will be explained later. The measure $\mu_n^{(s)}$ is finite if $s>-1$ and infinite if $s\leq -1$. By definition, the measure ${\mu}_n^{(s)}$ is invariant under the actions of the unitary group $U(n)$ by multiplication on the left and on the right.

If the constants $\const_{n,s}$ are chosen appropriately, then  the Pickrell family of measures has
the Kolmogorov property of consistency  under natural projections:  the push-forward of the Pickrell measure
${\mu}_{n+1}^{(s)}$ under the natural projection of cutting the $n\times n$-corner of a $(n+1)\times (n+1)$-matrix is precisely the Pickrell measure ${\mu}_n^{(s)}$. This consistency property is also verified for infinite Pickrell measures provided $n$ is sufficiently large;
see Proposition \ref{rel-consis} for the precise formulation. The consistency property and the Kolmogorov Theorem allows one  to define the Pickrell family of
measures $\mu^{(s)}$, $s\in {\mathbb R}$, on the space of infinite complex matrices.
The Pickrell measures are invariant by the action of the infinite unitary group on the left and on the right, and
the Pickrell family of measures is the natural analogue, in infinite dimension, of the canonical unitarily-invariant measure on a Grassmann  manifold, see Pickrell \cite{Pickrell3}.

What is the ergodic decomposition of Pickrell measures with respect to the action of the Cartesian square of the infinite unitary group?
The ergodic unitarily-invariant probability measures on the space of infinite complex matrices
 admit an explicit classification
due to Pickrell \cite{Pickrell1} and to which Olshanski and Vershik \cite{OlshVershik} gave a different approach:
each ergodic measure is determined by an infinite array $x=(x_1, \dots, x_n, \dots)$ on the half-line,
satisfying $x_1\geq x_2\dots\geq 0$ and $x_1+\dots+x_n+\dots<+\infty$, and an additional parameter
${\tilde \gamma}$ that we call the {\it Gaussian parameter}. Informally, the parameters $x_n$ should be thought of as ``asymptotic singular values'' of an infinite complex matrix, while ${\tilde \gamma}$ is the difference between the ``asymptotic trace'' and the sum of asymptotic eigenvalues (this difference is positive, in particular, for a Gaussian measure).

Borodin and Olshanski \cite{BO} proved in 2000 that for finite Pickrell measures the Gaussian parameter vanishes almost surely, and the ergodic decomposition measure, considered as a measure on the space of configurations on the half-line $(0, +\infty)$, coincides with the Bessel point
process of Tracy and Widom  \cite{TracyWidom}, whose correlation functions are given as determinants
of the Bessel kernel.

Borodin and Olshanski \cite{BO} posed the  problem:
{\it Describe the ergodic decomposition of infinite Pickrell measures.}
This paper gives a solution to the problem of Borodin and Olshanski.

The first step  is the result of \cite{Buf-inferg} that almost all ergodic components of an infinite Pickrell measure are themselves {\it finite}: only the decomposition measure itself is infinite. Furthermore, it will develop that, just as for finite measures, the Gaussian parameter vanishes. The ergodic decomposition measure can thus be identified with a sigma-finite measure ${\mathbb B}^{(s)}$ on the space of configurations on the half-line  $(0, +\infty)$.

How to describe a sigma-finite measure on the space of configurations? Note that the formalism of correlation functions is completely inapplicable, since these can only be defined for a finite measure.

This paper gives, for the first time, an explicit method for constructing infinite measures on spaces of configurations; since these measures are very closely related to determinantal probability measures, they are called {\it infinite determinantal measures}.

We give  three descriptions of the measure ${\mathbb B}^{(s)}$;  the first two can be carried out in much greater generality.

\begin{enumerate}
\item {\it Inductive limit of determinantal measures.} By definition, the measure ${\mathbb B}^{(s)}$ is supported on the set of configurations $X$
whose particles only  accumulate to zero, not to infinity. ${\mathbb B}^{(s)}$-almost every
 configuration $X$ thus admits a maximal particle $x_{max}(X)$. Now, if one takes an arbitrary $R>0$ and restricts the measure ${\mathbb B}^{(s)}$ onto the set $\{X: x_{max}(X)<R\}$, then
the resulting  restricted measure is finite and, after normalization, determinantal. The corresponding operator is an orthogonal projection operator whose range is found explicitly for any $R>0$.   The measure ${\mathbb B}^{(s)}$
is thus obtained as an inductive limit of finite determinantal measures along an exhausting family of subsets of the space of configurations.

\item {\it A determinantal measure times a multiplicative functional.}
More generally, one  reduces the measure ${\mathbb B}^{(s)}$ to a finite determinantal measure
 by taking the product with a suitable multiplicative functional. A {\it multiplicative functional} on the space of configurations is obtained by taking the product of the values of a fixed nonnegative function over all particles of a configuration:
$$
\Psi_g(X)=\prod\limits_{x\in X} g(x).
$$
If $g$ is suitably chosen, then the measure
\begin{equation}\label{psigbsdef}
\Psi_g{\mathbb B}^{(s)}
\end{equation}
is finite and, after normalization, determinantal. The corresponding operator is an orthogonal projection operator whose range is found explicitly. Of course, the previous description is a particular case of this one with $g=\chi_{(0, R)}$. It is often convenient to take a positive function, for example, the function $g^{\beta}(x)=\exp(-\beta x)$ for $\beta>0$.  While the range of the orthogonal projection operator inducing the measure \eqref{psigbsdef} is found explicitly for a wide class of functions $g$, it seems possible  to give a formula for its kernel for only very few functions; these computations will appear in the sequel to this paper.

\item  {\it A skew-product. } As was noted above,  ${\mathbb B}^{(s)}$-almost every configuration $X$
 admits a maximal particle $x_{max}(X)$, and it is natural to consider conditional measures of the measure
 ${\mathbb B}^{(s)}$ with respect to the position of the maximal particle $x_{max}(X)$. One obtains  a  well-defined determinantal probability measure induced by a projection operator  whose range, again, is found explicitly using the description of Palm measures of determinantal point processes due to Shirai and Takahashi \cite{ShirTaka2}. The sigma-finite distribution of the maximal particle is also explicitly found:
the ratios of the measures of intervals are obtained as ratios of suitable Fredholm determinants.
The measure ${\mathbb B}^{(s)}$ is thus represented as a skew-product whose base is
an explicitly found  sigma-finite measure on the half-line, and whose fibres are explicitly found determinantal probability measures. See Subsection~\ref{subsec-skewproduct} for a detailed presentation.
\end{enumerate}

The key r{\^o}le in the construction of infinite determinantal measures is played by the result of \cite{Buf-umn} (see also \cite{Buf-CIRM}) that a determinantal probability measure times an integrable  multiplicative functional is, after normalization, again a determinantal probability measure whose operator is found explicitly. In particular, if $\Prob_{\Pi}$ is a determinantal point process induced by a projection operator $\Pi$ with range $L$,
then, under certain additional assumptions, the measure  $\Psi_g\Prob_{\Pi}$ is, after normalization, a determinantal point process induced by the projection operator onto the subspace $\sqrt{g}L$; the precise statement is recalled in Proposition \ref{pr1-bis}.

Informally, if $g$ is such that the subspace $\sqrt{g}L$  no longer lies in $L_2$, then the measure $\Psi_g\Prob_{\Pi}$ ceases to be finite, and one obtains, precisely, an infinite determinantal measure corresponding to a subspace of locally square-integrable functions, one of the main constructions of this paper, see Theorem \ref{infdet-he}.

The Bessel point process of Tracy and Widom, which governs the ergodic decomposition of finite Pickrell measures, is the scaling limit of Jacobi orthogonal polynomial ensembles. In the problem of ergodic decomposition of infinite Pickrell measures one is led to investigating the scaling limit of infinite analogues of
Jacobi orthogonal polynomial ensembles. The resulting scaling limit, an infinite determinantal measure, is computed in the paper and called the infinite Bessel point process;
see Subsection~\ref{subsec-infbessel} for the precise definition.

The main result of the paper, Theorem \ref{mainthm}, identifies the ergodic decomposition measure of an infinite Pickrell measure with the infinite Bessel point process.

\subsection{Historical remarks}
Pickrell measures were introduced by Pickrell \cite{Pickrell3}  in 1987.
Borodin and Olshanski \cite{BO} studied in 2000 a closely related two-parameter family of  measures on the
space of infinite Hermitian matrices invariant  with respect to the natural action of the infinite unitary group by conjugation;
since the existence of such measures, as well as that of the original family considered by Pickrell,  is proved by a computation that goes back to the work of Hua Loo-Keng \cite{Hua}, Borodin and Olshanski gave to the measures of their family the name of {\it Hua-Pickrell measures}. For various generalizations of Hua-Pickrell measures, see e.g. Neretin \cite{Neretin}, Bourgade-Nikehbali-Rouault \cite{BNR}.
While Pickrell only considered values of the parameter for which the resulting measures are finite,
Borodin and Olshanski \cite{BO} showed that the inifnite Pickrell and Hua-Pickrell measures are also well-defined.
Borodin and Olshanski \cite{BO} proved that the ergodic decomposition of Hua-Pickrell probability measures  is given by determinantal point processes that arise as scaling limits of pseudo-Jacobian orthogonal polynomial ensembles and posed the problem of describing the ergodic decomposition of infinite Hua-Pickrell measures.

The aim of this paper, devoted to Pickrell's original model, is to give an explicit description for the ergodic decomposition of infinite Pickrell measures on spaces of infinite complex matrices.

\subsection{Organization of the paper}

This paper is the first of the cycle of three papers giving the explicit construction of the ergodic decomposition of infinite Pickrell measures.
Quotes to the other parts of the paper \cite{infdet2, infdet3} are organized as follows: Proposition II.2.3, equation (III.9), etc.

The paper is organized as follows.
In the Introduction, we proceed by illustrating the main construction of the paper, that of infinite determinantal measures, on the specific example of the infinite Bessel point process. Next we recall the construction of Pickrell measures and the Olshanski-Vershik approach to Pickrell's classification of ergodic unitarily-invariant measures on the space of infinite complex matrices. We then formulate the main result of the paper, Theorem \ref{mainthm}, which identifies the ergodic decomposition measure of an infinite Pickrell measure with the infinite Bessel point process
(subject to the change of variable $y=4/x$). We conclude the Introduction by giving an outline of the proof of Theorem \ref{mainthm}: the ergodic decomposition measures of  Pickrell measures are obtained as scaling limits
of their finite-dimensional approximations, the radial parts of finite-dimensional projections of Pickrell measures.  First, Lemma \ref{main-lemma} states that the rescaled radial parts, multiplied by a certain positive  density,  converge to the desired ergodic decomposition measure
multiplied by the same density.  Second, it will develop that the normalized
products of the push-forwards of rescaled radial parts to the space of configurations on the half-line with a suitably chosen multiplicative functional on the space of configurations, converge weakly in the space of measures on the space of configurations. Combining these statements will allow to conclude the proof of Theorem \ref{mainthm}.

Section 2 is devoted to the general construction of infinite determinantal measures on the space $\Conf (E)$ of configurations on a locally compact complete metric space $E$ endowed with a sigma-finite Borel measure $\mu$.

Start with a space $H$ of functions on $E$ locally square-integrable with respect to $\mu$ and an increasing collection of subsets
$$
E_0 \subset E_1 \subset \dots \subset E_n \subset \dots
$$
in $E$ such that for any $n \in \mathbb{N}$ the restricted subspace $\chi_{E_n}H$ is a closed subspace in $L_2(E,\mu)$. If the corresponding projection operator $\Pi_n$ is locally-trace class, then, by the Macch{\`\i}-Soshnikov Theorem, the projection operator $\Pi_n$ induces a determinantal measure $\Prob_n$ on $\Conf(E)$. Under certain additional assumptions  it follows from the result of \cite{Buf-umn} (see Corollary \ref{indsubset}) that the measures $\Prob_n$ satisfy the following consistency property: if $\Conf(E,E_n)$ stands for the subset of those configurations all whose particles lie in $E_n$, then for any $n \in \mathbb{N}$ we have
\begin{equation} \label{en-cons}
\frac{ \left. \Prob_{n+1} \right|_{\Conf(E,E_n)} }{\Prob_{n+1}(\Conf(E,E_n))} = \Prob_n
\end{equation}
The consistency property \eqref{en-cons} implies that there exists a sigma-finite measure $\mathbb{B}$ such that for any $n \in \mathbb{N}$ we have $0 < \mathbb{B}(\Conf(E,E_n)) < +\infty$ and
$$
\frac{ \left. \mathbb{B} \right|_{\Conf(E,E_n)} }{ \mathbb{B}(\Conf(E,E_n))} = \Prob_n
$$
The measure $\mathbb{B}$ is called an infinite determinantal measure. An alternative description of infinite determinantal measures uses the formalism of multiplicative functionals. In \cite{Buf-umn} it is proved (see also \cite{Buf-CIRM} and  Proposition \ref{pr1-bis}) that a determinantal measure times an integrable multiplicative functional is, after normalization, again a determinantal measure. Now, if one takes the product of a determinantal measure by a convergent, but not integrable, multiplicative functional, then one obtains an infinite determinantal measure. This reduction of infinite determinantal measure to usual ones by taking the product with a multiplicative functional is essential for the proof of Theorem \ref{mainthm}. Section 2 is concluded by the proof of the existence of the infinite Bessel point process.

The paper has three appendices. In Appendix~A, we collect the needed facts about the Jacobi orthogonal polynomials, including the recurrence relation between the $n$-th Christoffel-Darboux kernel corresponding to parameters $(\alpha, \beta)$ and the $n-1$-th Christoffel-Darboux kernel corresponding to parameters $(\alpha+2, \beta)$.

Appendix B is devoted to determinantal point processes on spaces of configurations. We start by recalling the definition of the space of configurations, its Borel structure and its topology; we next introduce determinantal point processes, recall the Macch{\`{\i}}-Soshnikov Theorem and the rule of transformation of kernels under a change of variables.
We next recall the definition of multiplicative functionals on the space of configurations, formulate the result of \cite{Buf-umn} (see also \cite{Buf-CIRM}) that a determinantal point process times a multiplicative functional is again a determinantal point process and give an explicit representation of the resulting kernels; in particular, we recall the representation  from \cite{Buf-umn}, \cite{Buf-CIRM}  for kernels of induced processes.

In Appendix C we recall the construction of Pickrell measures following a computation of Hua Loo-Keng \cite{Hua} as well as the observation  of Borodin and Olshanski \cite{BO} in the infinite case.

\subsection{The Infinite Bessel Point Process}
\label{subsec-infbessel}
\subsubsection{Outline of the construction}
Take $n\in {\mathbb N}$, $s\in {\mathbb R}$,  and endow the cube $(-1,1)^n$
with the measure
\begin{equation}
\label{sjacens}
\prod\limits_{1\leq i<j\leq n} (u_i-u_j)^2 \prod\limits_{i=1}^n (1-u_i)^sdu_i.
\end{equation}

For $s>-1$, the measure \eqref{sjacens} is a particular case of the Jacobi orthogonal polynomial ensemble, a determinantal point process induced by the $n$-th Christoffel-Darboux projection operator for Jacobi polynomials.
The classical Heine-Mehler asymptotics of Jacobi polynomials  implies
an asymptotics for the Christoffel-Dabroux kernels and, consequently, also for the corresponding determinantal point processes,
whose scaling limit,
with respect to the scaling
\begin{equation}\label{scaling}
u_i=1-\frac{y_i}{2n^2}, i=1, \dots, n,
\end{equation}
 is the Bessel point process of Tracy and Widom \cite{TracyWidom}. Recall  that the Bessel point process is governed by the projection operator, in $L_2((0, +\infty), \Leb)$,
onto the subspace of functions whose Hankel transform is supported in $[0,1]$.

For $s\leq -1$, the measure \eqref{sjacens} is infinite. To describe its scaling limit,
we start by recalling a recurrence relation between Christoffel-Darboux kernels of Jacobi polynomials and the consequent relation between the corresponding orthogonal polynomial ensembles: namely, the $n$-th Christoffel-Darboux kernel of the  Jacobi ensemble
with parameter $s$ is a rank one perturbation of the $n-1$-th  Christoffel-Darboux kernel of the Jacobi ensemble corresponding to parameter $s+2$.

This recurrence relation motivates the following construction.  Consider the range of the Christoffel-Darboux projection operator. It is a finite-dimensional subspace of polynomials of degree less than $n$ multiplied by the weight
$(1-u)^{s/2}$. Consider the same subspace for $s\leq -1$. The resulting space is no longer a subspace of $L_2$; it is nonetheless a well-defined space of {\it locally}
square-integrable functions. In view of the recurrence relation, furthermore, our subspace corresponding to the parameter $s$  is a  rank one perturbation of a similar subspace corresponding to parameter $s+2$, and so on, until we arrive at a value of the parameter, denoted $s+2n_s$ in what follows, for which the subspace becomes part of $L_2$. Our initial subspace is thus a finite-rank perturbation of a closed subspace in $L_2$ such that the rank of the perturbation depends on $s$ but not on $n$. Now we take this representation to the scaling limit and obtain
a subspace of locally square-integrable functions on $(0, +\infty)$,
which, again,  is a finite-rank perturbation of the range of the Bessel projection operator
corresponding to the parameter $s+2n_s$.

 To such a  subspace of locally square-integable functions we next assign a sigma-finite measure on the space of configurations, the {\it infinite Bessel point process}.
The infinite Bessel point process is the scaling limit  of the measures \eqref{sjacens}
under the scaling \eqref{scaling}.

\subsubsection{The Jacobi orthogonal polynomial ensemble}

First let $s>-1$.  Let $P_n^{(s)}$ be the
standard Jacobi orthogonal polynomials corresponding to
the weight $(1-u)^s$. Let $\tilde K_n^{(s)} (u_1, u_2)$ the $n$-th Christoffel-Darboux kernel of the Jacobi orthogonal  polynomial ensemble, see formulas
\eqref{CDJ}, \eqref{CDJ2} in the Appendix~A.
We now have the following well-known determinantal representation for the measure \eqref{sjacens} in the case $s>-1$:
\begin{equation}\label{det-CDJ}
\displaystyle{\const_{n, s}\prod\limits_{1\leq i<j\leq n} (u_i-u_j)^2 \prod\limits_{i=1}^n (1-u_i)^sdu_i=
\frac{1}{n!} \det \tilde K_{n}^{(s)} (u_i, u_j) \cdot \prod \limits_{i=1}^n du_i},
\end{equation}
where the normalization constant $\const_{n, s}$ is chosen in such a way that the left-hand side be a probability measure .

\subsubsection{The recurrence relation for Jacobi orthogonal polynomial ensembles}

We write $\Leb$ for the usual Lebesgue measure on the real line or on its subset.
Given a finite family of functions $f_1, \dots, f_N$ on the real line, let $\Span(f_1, \dots, f_N)$ stand for the vector space these functions span.
The Christoffel-Darboux kernel $\tilde K_{n}^{(s)}$ is the kernel of the operator of orthogonal projection, in the space $L_2([-1,1], \Leb)$,
onto the subspace
\begin{multline*}
L_{Jac}^{(s,n)}=\Span\left((1-u)^{s/2}, (1-u)^{s/2}u
, \dots, (1-u)^{s/2}u^{n-1}\right)=\\
=\Span\left((1-u)^{s/2}, (1-u)^{s/2+1}
, \dots, (1-u)^{s/2+n-1}\right).
\end{multline*}
By definition, we have a direct-sum decomposition
$$
L_{Jac}^{(s,n)}={\mathbb C}(1-u)^{s/2}\oplus  L_{Jac}^{(s+2,n-1)}
$$
By Proposition \ref{rec-form-jac},  for any $s>-1$ we have the recurrence relation
\begin{multline*}
\tilde K_{n}^{(s)}(u_1, u_2)=
\frac{s+1}{2^{s+1}}P_{n-1}^{(s+1)}(u_1)(1-u_1)^{s/2} P_{n-1}^{(s+1)}(u_2)(1-u_2)^{s/2}+\\+ \tilde K_{n}^{(s+2)}(u_1, u_2)
\end{multline*}
and, consequently, an orthogonal direct-sum decomposition
$$
L_{Jac}^{(s,n)}={\mathbb C}P_{n-1}^{(s+1)}(u) (1-u)^{s/2}\oplus  L_{Jac}^{(s+2,n-1)}.
$$
We now pass to the case $s\leq -1$.
Define a natural number $n_s$  by the relation
$$
\frac{s}{2}+n_s\in\left(-\frac12, \frac12\right]
$$
and introduce the subspace
\begin{equation}\label{pertus}
\tilde V^{(s,n)}=\Span\left( (1-u)^{s/2},  (1-u)^{s/2+1}, \dots,        P_{n-n_s}^{(s+2n_s-1)}(u) (1-u)^{s/2+n_s-1}\right).
\end{equation}

By definition, we have
is a direct sum decomposition
\begin{equation}\label{ljsns}
L_{Jac}^{(s,n)}=\tilde V^{(s,n)}\oplus L_{Jac}^{(s+2n_s, n-n_s )}.
\end{equation}
Note here that
$$
L_{Jac}^{(s+2n_s, n-n_s )}\subset L_2([-1,1], \Leb),
$$
while
$$
\tilde V^{(s,n)}\cap  L_2([-1,1], \Leb)=0.
$$

\subsubsection{Scaling limits}
Recall  that the scaling limit, with respect to the scaling \eqref{scaling},  of Christoffel-Darboux kernels  ${\tilde K}_n^{(s)}$ of the Jacobi orthogonal polynomial ensemble, is given by the Bessel kernel ${\tilde J}_s$ of Tracy and Widom \cite{TracyWidom} (the definition of the Bessel kernel is recalled in the Appendix~A and the precise statement on the scaling limit is recalled in Proposition \ref{tilde-kernel-kns-js}).

It is clear that, for any $\beta$, under the scaling \eqref{scaling}, we have
$$
\lim\limits_{n\to\infty} (2n^2)^{\beta} (1-u_i)^{\beta}=y_i^{\beta}
$$
and, for any $\alpha>-1$, by the classical Heine-Mehler asymptotics for Jacobi polynomials, we have
$$
\lim\limits_{n\to\infty}2^{-\frac{\alpha+1}{2}}n^{-1}P_{n}^{(\alpha)}(u_i) (1-u_i)^{\frac{\alpha-1}{2}}=
\frac{ J_{\alpha}(\sqrt{y_i})}{\sqrt{y_i}}.
$$

It is therefore natural  to take the subspace
\begin{equation*}
{\tilde V}^{(s)}=\Span \left(y^{s/2}, y^{s/2+1}, \dots,  \frac{J_{s+2n_s-1}(\sqrt{y})}{\sqrt{y}}\right).
\end{equation*}
as the scaling limit of the subspace \eqref{pertus}.

Furthermore, we  already know that
the scaling limit  of the subspace \eqref{ljsns} is the subspace ${\tilde L}^{(s+2n_s)}$, the range
of the operator ${\tilde J}_{s+2n_s}$.

We arrive at the subspace ${\tilde H}^{(s)}$
\begin{equation*}
{\tilde H}^{(s)}={\tilde V}^{(s)}\oplus {\tilde L}^{(s+2n_s)}.
\end{equation*}
It is natural to consider  the subspace ${\tilde H}^{(s)}$ as the scaling limit of the subspaces $L_{Jac}^{(s,n)}$
under the scaling \eqref{scaling} as $n\to\infty$.

Note that the subspace ${\tilde H}^{(s)}$ consists of locally square-integrable functions, which, moreover, only fail to be square-integrable {\it at zero}: for any $\varepsilon>0$, the subspace
$\chi_{[\varepsilon, +\infty)} {\tilde H}^{(s)}$ is contained in $L_2$.

\subsubsection{Definition of the infinite Bessel point process}

We now proceed  to a precise description, in this specific case, of one of the main constructions of the paper: that of a sigma-finite measure
${\tilde {\mathbb B}}^{(s)}$, the scaling limit of infinite Jacobi ensembles \eqref{sjacens} under the scaling \eqref{scaling}.
Let $\Conf((0, +\infty))$ be the space of configurations on $(0, +\infty)$.
Given a Borel subset $E_0\subset (0, +\infty)$, we let $\Conf((0, +\infty), E_0)$ be the subspace of
configurations all whose particles lie in $E_0$.
Generally, given a measure $\mathbb{B}$ on a set $X$ and a measurable subset $Y\subset X$ such
that $0<\mathbb{B}(Y)<+\infty$, we let $\mathbb{B}|_Y$ stand for the
restriction of the measure $\mathbb{B}$ onto the subset $Y$.

 It will be proved in what follows that, for any $\varepsilon>0$, the subspace
$\chi_{(\varepsilon, +\infty)}{\tilde H}^{(s)}$ is a closed subspace of $L_2((0, +\infty), \Leb)$
 and  that the operator ${\tilde \Pi}^{(\varepsilon, s)}$
of orthogonal projection onto the subspace $\chi_{(\varepsilon, +\infty)}{\tilde H}^{(s)}$ is locally of trace class. By the Macch{\`i}-Soshnikov Theorem, the operator ${\tilde \Pi}^{(\varepsilon, s)}$ induces a determinantal measure $\Prob_{ {\tilde \Pi}^{(\varepsilon, s)}  }$ on $\Conf((0, +\infty))$.

\begin{proposition}\label{inf-bessel} Let $s\leq -1$. There exists a sigma-finite measure
$\mathbb{B}^{(s)}$ on $\Conf((0, +\infty))$ such that
we have
\begin{enumerate}
\item the particles of ${\mathbb B}$-almost every configuration do not accumulate at zero;
\item  for any $\varepsilon>0$
 we have $$0<\mathbb{B}(\Conf((0, +\infty); (\varepsilon, +\infty))<+\infty$$
and  $$
\frac {\mathbb{B}|_{\Conf((0, +\infty); (\varepsilon, +\infty))}}
{\mathbb{B}(\Conf((0, +\infty); (\varepsilon, +\infty))}=\Prob_{{\tilde \Pi}^{(\varepsilon, s)}}.$$

\end{enumerate}
\end{proposition}

These conditions define the measure ${\tilde {\mathbb{B}}^{(s)}}$ uniquely up to multiplication by a constant.

{\bf Remark.} For $s\neq -1, -3, \dots$, we can also write
$$
{\tilde H}^{(s)}=\Span(y^{s/2}, \dots, y^{s/2+n_s-1})\oplus {\tilde L}^{(s+2n_s)}
$$
and use the preceding construction otherwise without change.
For $s=-1$ note that the function $y^{1/2}$ fails to be square-integrable at infinity --- whence the need for the definition given above.
For $s>-1$, write ${\tilde {\mathbb B}}^{(s)}=\Prob_{{\tilde J}_{s}}$.
\begin{proposition}\label{tildebs-sing}
If $s_1\neq s_2$, then the measures ${\tilde {\mathbb B}}^{(s_1)}$ and ${\tilde {\mathbb B}}^{(s_2)}$ are mutually singular.
\end{proposition}
 The proof of Proposition \ref{tildebs-sing} will be derived from Proposition \ref{bs-sing}, which in turn, will be obtained as a corollary of the main result, Theorem \ref{mainthm}.

\subsection{The modified Bessel point process}

In what follows, we will need the Bessel point process subject to the change of variable $y=4/x$.  We thus consider the half-line $(0, +\infty)$ endowed with the standard Lebesgue measure
$\Leb$.
Take $s>-1$ and introduce  a kernel $J^{(s)}$  by the formula
\begin{multline*}
J^{(s)}(x_1,x_2) = \frac{J_s\left(\frac{2}{\sqrt{x_1}}\right)\frac{1}{\sqrt{x_2}}J_{s+1}\left(\frac{2}{\sqrt{x_2}}\right)- J_s\left(\frac{2}{\sqrt{x_2}}\right)\frac{1}{\sqrt{x_1}}J_{s+1}\left(\frac{2}{\sqrt{x_1}}\right)}{x_1-x_2}, \\ x_1>0,\;x_2>0\,.
\end{multline*}
or, equivalently,
\begin{equation*}
J^{(s)}(x_1,x_2) = \frac1{x_1x_2}\int\limits_0^1  J_s\left(2\sqrt{\frac{t}{x_1}}\right)
J_s\left(2\sqrt{\frac{t}{x_2}}\right)dt.
\end{equation*}

The change of variable $y=4/x$ reduces the kernel $J^{(s)}$ to the
 kernel ${\tilde J}_{s}$ of  the Bessel point process
of Tracy and Widom considered above (recall here that a change of variables $u_1=\rho(v_1)$, $u_2=\rho(v_2)$ transforms a kernel $K(u_1, u_2)$
to a kernel of the form $K(\rho(v_1), \rho(v_2))\sqrt{\rho'(v_1)\rho'(v_2)}$).
The kernel $J^{(s)}$ therefore  induces on the space $L_2((0,+\infty), \Leb)$
a locally trace class operator of orthogonal projection, for which, slightly abusing notation, we keep the symbol $J^{(s)}$; we denote ${ L}^{(s)}$ the range of $J^{(s)}$.
 By the Macch\`{\i}-Soshnikov Theorem, the operator $J^{(s)}$
 induces a determinantal measure $\Prob_{J^{(s)}}$ on $\Conf((0,+\infty))$.

\subsection{The modified infinite  Bessel point process}
The involutive homeomorphism
$$y=4/x$$ of the half-line $(0, +\infty)$ induces a corresponding change of variable homeomorphism of the space $\Conf((0, +\infty))$.
Let $\mathbb{B}^{(s)}$ be the image of $\tilde {\mathbb{B}}^{(s)}$ under our change of variables.
As we shall see below, the measure
$\mathbb{B}^{(s)}$ is precisely the ergodic decomposition measure for the infinite Pickrell measures.

A more  explicit description of the measure $\mathbb{B}^{(s)}$ can be given as follows.

By definition, we have
$$
L^{(s)}=\left\{\frac{\varphi(4/x)}{x}, \varphi\in {\tilde L}^{(s)}\right\}.
$$
(the behaviour of determinantal measures under a change of variables is recalled in the Subsection~\ref{subsec-ch-of-var}).

We similarly let $V^{(s)}, H^{(s)}\subset L_{2, \loc}((0, +\infty), \Leb)$ be the images of the subspaces ${\tilde V}^{(s)}$,  ${\tilde H}^{(s)}$
under our change of variables $y=4/x$:

$$V^{(s)}=\left\{\frac{\varphi(4/x)}{x}, \varphi\in {\tilde V}^{(s)}\right\}, \
H^{(s)}=\left\{\frac{\varphi(4/x)}{x}, \varphi\in {\tilde H}^{(s)}\right\}.
$$
 By definition,  we have
\begin{gather*}
V^{(s)}=\Span\left(x^{-s/2-1}, \dots, \frac{J_{s+2n_s-1}(\frac2{\sqrt{x}})}{\sqrt{x}}\right),\\
H^{(s)}=V^{(s)}\oplus L^{(s+2n_s)}.
\end{gather*}

It will develop that for all  $R>0$ the subspace $\chi_{(0, R)}H^{(s)}$ is a closed subspace in
$L_2((0, +\infty), \Leb)$; let $\Pi^{(s, R)}$ be the corresponding orthogonal projection operator.
By definition, the operator $\Pi^{(s, R)}$ is locally of trace-class and, by the Macch\`{\i}-Soshnikov Theorem, the operator $\Pi^{(s, R)}$
 induces a determinantal measure $\Prob_{\Pi^{(s, R)}}$ on $\Conf((0,+\infty))$.

The measure $\mathbb{B}^{(s)}$ is
 characterized  by the following conditions:
\begin{enumerate}
\item the set of  particles of $\mathbb{B}^{(s)}$-almost every configuration  is bounded;
\item for all $R>0$
 we have $$0<\mathbb{B}(\Conf((0, +\infty); (0, R))<+\infty$$
and  $$
\frac {\mathbb{B}|_{\Conf((0, +\infty); (0, R))}}
{\mathbb{B}(\Conf((0, +\infty); (0, R))}=\Prob_{{ \Pi}^{(s, R)}}.$$
\end{enumerate}
These conditions define the measure ${ {\mathbb{B}}^{(s)}}$ uniquely up to multiplication by a constant.

{\bf Remark.} For $s\neq -1, -3, \dots$, we can of course also write
$$
H^{(s)}=\Span(x^{-s/2-1}, \dots, x^{-s/2-n_s+1})\oplus L^{(s+2n_s)}.
$$
Let $\scrI_{1, \loc}((0, +\infty), \Leb)$ be the space of locally trace-class operators acting on the space  $L_2((0, +\infty), \Leb)$ (see Subsection~\ref{subsec-loc-tr-cls} for the detailed definition).
We have the following proposition describing the asymptotic behaviour of the operators ${ \Pi}^{(s, R)}$
as $R\to\infty$.
\begin{proposition}\label{asympt-R-bessel}
Let $s\leq -1$. Then
\begin{enumerate}
\item as $R\to\infty$ we have $${ \Pi}^{(s, R)}\to J^{(s+2n_s)}$$ in $\scrI_{1, \loc}((0, +\infty), \Leb)$;
\item Consequently, as $R\to\infty$, we have $$\Prob_{{ \Pi}^{(s, R)}}\to \Prob_{J^{(s+2n_s)}}$$
weakly in the space of probability measures on $\Conf((0, +\infty))$.
\end{enumerate}

\end{proposition}

As before, for $s>-1$, write ${ {\mathbb B}}^{(s)}=\Prob_{{ J}^{(s)}}$.
Proposition \ref{tildebs-sing} is equivalent to the following
\begin{proposition}\label{bs-sing}
If $s_1\neq s_2$, then the measures ${ {\mathbb B}}^{(s_1)}$ and ${ {\mathbb B}}^{(s_2)}$ are mutually singular.
\end{proposition}
 Proposition \ref{bs-sing}  will be obtained as the corollary of the main result, Theorem \ref{mainthm}, in the last section of the paper.

We now represent the measure ${\mathbb B}^{(s)}$ as the product of a
determinantal probability measure and a multiplicative functional.
Here we limit ourselves to  specific example of such a representation, but in what follows
we will see that they can be constructed in much greater generality.
Introduce a function $S$ on the space $\Conf((0, +\infty))$ by setting
$$
S(X)=\sum\limits_{x\in X} x.
$$
The function $S$ may, of course, assume value $\infty$, but the set of such configurations is ${\mathbb B}^{(s)}$-negligible, as is shown by the following
\begin{proposition}\label{expbetafin}
For any $s\in {\mathbb R}$ we have $S(X)<+\infty$ almost surely with respect to the measure ${\mathbb B}^{(s)}$ and for any $\beta>0$ we have
$$
\exp(-\beta S(X))\in L_1(\Conf((0, +\infty)), {\mathbb B}^{(s)}).
$$
\end{proposition}
Furthermore, we shall now see that the measure
$$
\frac{\exp(-\beta S(X)){\mathbb B}^{(s)}}{\displaystyle \int \limits_{\Conf((0, +\infty))}\exp(-\beta S(X))d{\mathbb B}^{(s)}}
$$
is determinantal.

\begin{proposition} \label{expbetal}
For any $s\in\mathbb{R}$, $\beta>0$,
 the subspace
\begin{equation}\label{expaa}
\exp\left(-\beta x/2\right)H^{(s)}
\end{equation}
is a closed subspace of $L_2\bigl((0,+\infty),\Leb\bigr)$, and the operator of orthogonal projection onto the subspace \eqref{expaa} is locally of trace class.
\end{proposition}

Let ${\Pi}^{(s,\beta)}$ be the operator of orthogonal projection onto the subspace \eqref{expaa}.

By Proposition \ref{expbetal} and  the Macch{\`\i}-Soshnikov Theorem, the operator ${\Pi}^{(s,\beta)}$
induces a determinantal probability measure on the space $\Conf((0, +\infty))$.
\begin{proposition}\label{expbetacharac}
For any $s\in\mathbb{R}$, $\beta>0$, we have
\begin{equation}\label{pisb}
\frac{\exp(-\beta S(X)){\mathbb B}^{(s)}}{\displaystyle \int \limits_{\Conf((0, +\infty))}\exp(-\beta S(X))d{\mathbb B}^{(s)}}=\Prob_{{\Pi}^{(s,\beta)}}.
\end{equation}
\end{proposition}

\subsection{Unitarily-Invariant  Measures on Spaces of Infinite Matrices}

\subsubsection{Pickrell Measures}\label{sec:Pickrell_measures}

 Let $\Mat(n, \mathbb C)$ be the space of $n \times n$ matrices with complex entries:
\[
\Mat (n, \mathbb C) = \{ z = (z_{ij}),\; i=1,\dots,n;j=1,\dots, n\}
\]

Let $\Leb=dz$ be the Lebesgue measure on $\Mat (n, \mathbb C)$.
For $n_1<n$, let $$\pi^n_{n_1}:\ \Mat(n,\mathbb{C})\to\Mat(n_1,\mathbb{C})$$ be the natural projection map that to a matrix ${z}=({z}_{ij}),i,j=1,\dots,n,$ assigns its upper left corner, the matrix $\pi_{n_1}^n({z})=({z}_{ij}),i,j=1,\dots,n_1.$

Following Pickrell \cite{Pickrell1}, take $s\in\mathbb{R}$ and introduce a measure $\widetilde{\mu}_n^{(s)}$ on $\Mat(n,\mathbb{C})$ by the formula $$\widetilde{\mu}_n^{(s)}=\det(1+{z}^*{z})^{-2n-s}dz.$$

The measure $\widetilde{\mu}_n^{(s)}$ is finite if and only  if $s>-1$.

The measures $\widetilde{\mu}_n^{(s)}$ have the following property of consistency with respect to the projections $\pi_{n_1}^n$.

\begin{proposition}\label{rel-consis}
Let $s\in{\mathbb R}$, $n\in {\mathbb N}$ satisfy $n+s>0$.
Then for any ${\tilde z}\in \Mat(n, {\mathbb C})$ we have
\begin{multline*}
\int\limits_{(\pi^{n+1}_n)^{-1}({\tilde z})}
\det(1+{z}^*{z})^{-2n-2-s}dz=\\
=\frac{\pi^{2n+1}(\Gamma(n+1+s))^2}{\Gamma(2n+2+s)\cdot\Gamma(2n+1+s)}
\det(1+{{\tilde z}}^*{{\tilde z}})^{-2n-s}.
\end{multline*}
\end{proposition}

Now let $\Mat(\mathbb{N},\mathbb{C})$ be the space of infinite matrices whose rows and columns are indexed by natural numbers and whose entries are complex:
$$\Mat(\mathbb{N},\mathbb{C})=\{z=(z_{ij}),i,j\in\mathbb{N},z_{ij}\in\mathbb{C}\}.$$
Let $\pi_n^{\infty}:\Mat(\mathbb{N},\mathbb{C})\to\Mat(n,\mathbb{C})$ be the natural projection map that to an infinite matrix $z\in \Mat(\mathbb{N},\mathbb{C})$ assigns its  upper left $n\times n$-``corner'', the matrix $(z_{ij}), i,j=1,\dots,n.$

For $s>-1$, Proposition \ref{rel-consis} together with the Kolmogorov Existence Theorem \cite{Kolmogorov} implies that there exists a unique probability measure $\mu^{(s)}$ on $\Mat(\mathbb{N},\mathbb{C})$ such that for any $n\in {\mathbb N}$ we have
the relation
$$(\pi_{n}^{\infty})_*\mu^{(s)}=\pi^{-n^2}\prod\limits_{l=1}^n
\frac{\Gamma(2l+s)\Gamma(2l-1+s)}{(\Gamma(l+s))^2}\widetilde{\mu}_n^{(s)}.$$

If $s\leq -1$, then  Proposition \ref{rel-consis} together with the Kolmogorov Existence Theorem \cite{Kolmogorov} implies that for any $\lambda>0$  there exists a unique infinite measure $\mu^{(s,\lambda)}$ on $\Mat(\mathbb{N},\mathbb{C})$ such that
\begin{enumerate}\item for any $n\in {\mathbb N}$ satisfying $n+s>0$
and any compact subset $Y\subset \Mat(n,\mathbb{C})$ we have
$\mu^{(s,\lambda)}(Y)<+\infty$; the pushforwards $(\pi_{n}^{\infty})_*\mu^{(s,\lambda)}$
are consequently well-defined;
\item for any $n\in {\mathbb N}$ satisfying $n+s>0$
we have
\begin{equation*}
(\pi_{n}^{\infty})_*\mu^{(s,\lambda)}=\lambda\left(\prod\limits_{l=n_0}^n\pi^{-2n}
\frac{\Gamma(2l+s)\Gamma(2l-1+s)}{(\Gamma(l+s))^2}\right)\widetilde{\mu}^{(s)}.
\end{equation*}
\end{enumerate}

The measures $\mu^{(s,\lambda)}$ will be called {\it infinite Pickrell measures}.
Slightly abusing notation, we shall omit the super-script $\lambda$ and write $\mu^{(s)}$ for a measure defined up to a multiplicative constant. See p.116 in Borodin and Olshanski \cite{BO} for a detailed presentation of infinite Pickrell measures.
\begin{proposition} \label{singul}
For any $s_1, s_2\in {\mathbb R}$, $s_1\neq s_2$,  the Pickrell measures $\mu^{(s_1)}$ and
$\mu^{(s_2)}$ are mutually singular.
\end{proposition}
Proposition \ref{singul}  is obtained from Kakutani's Theorem in the spirit of \cite{BO}, see also \cite{Neretin}.

Let $U(\infty)$ be the infinite unitary group: an infinite matrix $u=(u_{ij})_{i,j\in {\mathbb N}}$ belongs to
$U(\infty)$ if there exists a natural number $n_0$ such that the matrix
$$
(u_{ij}), i,j\in [1,n_0]
$$
is unitary, while $u_{ii}=1$ if $i>n_0$ and $u_{ij}=0$ if $i\neq j$, $\max(i,j)>n_0$.

The group $U(\infty) \times U(\infty)$ acts on $\Mat({\mathbb N}, {\mathbb C})$
by multiplication on both sides:
\begin{equation*}
T_{(u_1,u_2)}z \;=\; u_1zu_2^{-1}.
\end{equation*}

The Pickrell measures $\mu^{(s)}$ are by definition $U(\infty)\times U(\infty)$-invariant.
For the r{\^o}le of Pickrell and related mesures in the representation theory of $U(\infty)$, see
\cite{Olsh1}, \cite{Olsh2}, \cite{OlshVershik}.

Theorem 1 and Corollary 1 in \cite{Buf-ergdec} imply that the measures $\mu^{(s)}$ admit an ergodic decomposition, while
Theorem 1 in \cite{Buf-inferg} implies that for any $s\in\mathbb{R}$ the ergodic components of the measure $\mu^{(s)}$ are almost surely finite. We now formulate this result in  greater detail. Recall that a $U(\infty)\times U(\infty)$-invariant probability measure
on $\Mat({\mathbb N}, {\mathbb C})$ is called {\it ergodic} if every $U(\infty)\times U(\infty)$invariant Borel subset of $\Mat({\mathbb N}, {\mathbb C})$ either has measure zero or has complement of measure zero. Equivalently, ergodic probability
measures are extremal points of the convex set of all $U(\infty)\times U(\infty)$-invariant probability measures on $\Mat({\mathbb N}, {\mathbb C})$.
Let ${\mathfrak M}_{\mathrm{erg}}(\Mat({\mathbb N}, {\mathbb C})$ stand for the set of all ergodic
$U(\infty)\times U(\infty)$-invariant probability measures on $\Mat({\mathbb N}, {\mathbb C})$. The set ${\mathfrak M}_{\mathrm{erg}}(\Mat({\mathbb N}, {\mathbb C}))$ is a Borel subset of
the set of all probability measures on $\Mat({\mathbb N}, {\mathbb C})$ (see, e.g., \cite{Buf-ergdec}). Theorem 1 in \cite{Buf-inferg} implies that for any $s\in\mathbb{R}$
there exists a unique sigma-finite Borel measure ${\overline \mu}^{(s)}$ on the set ${\mathfrak M}_{\mathrm{erg}}(\Mat({\mathbb N}, {\mathbb C}))$ such that we have
\begin{equation*}
\mu^{(s)}=\int\limits_{{\mathfrak M}_{\mathrm{erg}}(\Mat({\mathbb N}, {\mathbb C})}\eta d{\overline \mu}^{(s)}(\eta).
\end{equation*}

 The main result of this paper is an explicit description of the measure ${\overline \mu}^{(s)}$
and its identification, after a change of variable,  with the infinite Bessel point process considered above.

\subsection{Classification of ergodic measures}

First, we recall the classification of ergodic probability $U(\infty)\times U(\infty)$-invariant measures on
$\Mat(\mathbb{N},\mathbb{C})$. This classification has been obtained
by Pickrell \cite{Pickrell1}, \cite{Pickrell2}; Vershik \cite{Vershik} and Olshanski and Vershik  \cite{OlshVershik} proposed a different approach to this classification in the case of unitarily-invariant measures on the space of infinite Hermitian matrices, and Rabaoui \cite{Rabaoui1}, \cite{Rabaoui2} adapted the Olshanski-Vershik approach to the initial problem of Pickrell.
In this note, the Olshanski-Vershik approach is followed as well.

Take $z\in \Mat(\mathbb{N},\mathbb{C})$, denote $z^{(n)}=\pi^{\infty}_{n}z,$
and let
\begin{equation*}
\lambda_1^{(n)}\geq\dots\geq\lambda_n^{(n)}\geq0
\end{equation*}
be the eigenvalues of
the matrix
$$
\left(z^{(n)}\right)^*z^{(n)},
$$
counted with multiplicities, arranged in non-increasing order.
To stress dependence on $z$, we write $\lambda_i^{(n)}=\lambda_i^{(n)}(z)$.

\begin{teorema}
\begin{enumerate}
\item Let $\eta$ be an ergodic Borel $U(\infty)\times U(\infty)$-invariant probability measure on $\Mat(\mathbb{N},\mathbb{C})$. Then there exist non-negative real numbers
$$\gamma\geq0,\ x_1\geq x_2\geq\dots\geq x_n\geq\dots\geq0,$$
satisfying $\displaystyle\gamma\geq\sum\limits_{i=1}^{\infty}x_i$, such that for $\eta$-almost every
$z\in\Mat(\mathbb{N},\mathbb{C})$ and any $i\in\mathbb{N}$ we have:
\begin{equation}\label{xgamma}
x_i=\lim\limits_{n\to\infty}\frac{\lambda_i^{(n)}(z)}{n^2},\ \
\gamma=\lim\limits_{n\to\infty}\frac{\tr\left(z^{(n)}\right)^*z^{(n)}}{n^2}.
\end{equation}
\item Conversely, given non-negative real numbers
$\gamma\geq0,\;\,x_1\geq x_2\geq\dots\geq x_n\geq\dots\geq0$ such that
$$\displaystyle\gamma\geq\sum\limits_{i=1}^{\infty}x_i\,,$$
there exists a unique $U(\infty)\times U(\infty)$-invariant ergodic Borel probability measure $\eta$ on
$\Mat(\mathbb{N},\mathbb{C})$ such that the relations \eqref{xgamma} hold for $\eta$-almost all $z\in \Mat(\mathbb{N},\mathbb{C})$.
\end{enumerate}
\end{teorema}

Introduce the {\it Pickrell set} $\Omega_P\subset\mathbb{R}_+\times\mathbb{R}_+^\mathbb{N}$ by the formula
$$\Omega_P=\left\{\omega=(\gamma,x)\colon x=(x_n),\;n\in\mathbb{N},\;x_n\geq x_{n+1}\geq0,\;
\gamma\geq\sum\limits_{i=1}^{\infty}x_i\right\}.$$
The set $\Omega_P$ is, by definition, a closed subset of $\mathbb{R}_+\times\mathbb{R}_+^\mathbb{N}$ endowed with the Tychonoff topology.
For $\omega\in\Omega_P$ we let $\eta_{\omega}$ be the corresponding ergodic probability measure.

The Fourier transform of the measure $\eta_{\omega}$ is explicitly described as follows.
 First,  for any $\lambda \in \mathbb{R}$ we have
\begin{equation}\label{fourier-la}
\int \limits _{\Mat(\mathbb{N}, \mathbb{C})}
 \exp(i\lambda \Re z_{11}) d\eta_{\omega} (z)=
\frac{ \exp(-4(\gamma - \sum \limits _{k=1}^{\infty}x_k)\lambda^2)}
 {\prod\limits_{k=1}^{\infty}(1+4x_k \lambda^2)}.
\end{equation}
Denote  $F_{\omega}(\la)$ the expression in the right-hand side of \eqref{fourier-la};
then, for any $\la_1, \dots, \la_m\in \mathbb{R}$ we have
\begin{equation*}
\int \limits _{\Mat(\mathbb{N}, \mathbb{C})}
 \exp(i(\lambda_1 \Re z_{11}+\dots+\la_m\Re z_{mm} )) d\eta_{\omega} (z)=
 F_{\omega}(\la_1)\cdot\dots\cdot F_{\omega}(\la_m).
\end{equation*}
The Fourier transform is fully defined, and the measure $\eta_{\omega}$ is completely described.
An explicit construction of the ergodic measures $\eta_{\omega}$ is given as follows. First, if one takes all entries of the matrix $z$ are independent identically distributed complex Gaussian random variables with expectation $0$ and
 variance ${\tilde \gamma}$, then the resulting Gaussian measure with parameter ${\tilde \gamma}$,  clearly unitarily invariant
and, by the Kolmogorov zero-one law, ergodic, corresponds to the parameter
$\omega=({\tilde \gamma}, 0, \dots, 0, \dots)$ --- all $x$-coordinates are equal to $0$ ( indeed,   singular values of a Gaussian matrix grow at rate $\sqrt{n}$ rather than $n$).

Next, let $(v_1, \dots, v_n, \dots)$, $(w_1, \dots, w_n, \dots)$ be two infinite independent vectors of independent identically distributed complex Gaussian random variables with variance $\sqrt{x}$, and set $z_{ij}=v_iw_j$. One thus obtains  a measure whose unitary invariance is clear and whose ergodicity is immediate from the Kolmogorov zero-one law.
This measure corresponds to the parameter $\omega\in\Omega_P$ such that $\gamma(\omega)=x$, $x_1(\omega)=x$, and all the other parameters are zero. Following Olshanski and Vershik \cite{OlshVershik}, such measures are called {\it Wishart measures} with parameter $x$.
 In the general case, set ${\tilde \gamma}=\gamma - \sum \limits _{k=1}^{\infty}x_k$.
The measure $\eta_{\omega}$ is then an infinite convolution of the Wishart measures with parameters $x_1, \dots, x_n, \dots$ and the Gaussian measure with parameter ${\tilde \gamma}$. Convergence of the series $x_1+\dots+x_n+\dots$ ensures that the convolution is well-defined.

The quantity ${\tilde \gamma}=\gamma - \sum \limits _{k=1}^{\infty}x_k$ will therefore be called the {\it Gaussian parameter} of the measure $\eta_{\omega}$.
It will develop that the Gaussian parameter vanishes for almost all ergodic components of Pickrell measures.

By Proposition 3 in \cite{Buf-ergdec}, the subset of ergodic $U(\infty)\times U(\infty)$-invariant measures
is a Borel subset of the space of all Borel probability measures on $\Mat(\mathbb{N},\mathbb{C})$ endowed
with the natural Borel structure (see, e.g., \cite{Bogachev}).
Furthermore, if one denotes $\eta_{\omega}$ the Borel ergodic probability measure corresponding to a point $\omega\in\Omega_P$, $\omega=(\gamma,x)$, then the correspondence
$$\omega\to\eta_{\omega}$$
is a Borel isomorphism of the Pickrell set $\Omega_P$ and the set of $U(\infty)\times U(\infty)$-invariant ergodic probability measures on $\Mat(\mathbb{N},\mathbb{C})$.

The Ergodic Decomposition Theorem (Theorem 1 and Corollary 1 of \cite{Buf-ergdec})  implies that
each Pickrell measure $\mu^{(s)}$, $s\in\mathbb{R}$, induces a unique decomposing measure $\overline{\mu}^{(s)}$ on $\Omega_P$ such that we have
\begin{equation}
\label{ergdecmubar}
\mu^{(s)}=\int\limits_{\Omega_P}\eta_{\omega}\,d\overline{\mu}^{(s)}(\omega)\;.
\end{equation}
The integral is understood in the usual weak sense, see \cite{Buf-ergdec}.

For $s>-1$, the measure $\overline{\mu}^{(s)}$ is a probability measure on $\Omega_P$,
while for $s\leq-1$ the measure $\overline{\mu}^{(s)}$ is infinite.

Set
$$\Omega_P^0=\{ (\gamma,\{x_n\})\in\Omega_P: x_n>0\text{\quad for all }n,\  \gamma=\sum\limits_{n=1}^{\infty} x_n \}.$$

The subset $\Omega_P^0$ is of course { not} closed in $\Omega_P$.

Introduce a map $$
\conf\colon\Omega_P\to  \Conf((0, +\infty))
$$
that to a point $\omega\in\Omega_P, \omega=(\gamma, \{x_n\})$ assigns
the configuration $$
\conf(\omega)=(x_1, \dots, x_n, \dots)\in
\Conf((0, +\infty)).
$$
The map $\omega\to \conf(\omega)$ is bijective in restriction to the subset
$\Omega_P^0$.

{\bf Remark.} In the definition of the map $\conf$, the ``asymptotic eigenvalues'' $x_n$ are  counted with multiplicities, while, if $x_{n_0}=0$ for some $n_0$, then $x_{n_0}$
and all subsequent terms are discarded, and the resulting configuration is finite. We shall see, however,
that, $\overline{\mu}^{(s)}$-almost surely, all configurations are infinite and  that,
$\overline{\mu}^{(s)}$-almost surely, all multiplicities are equal to one. It will also develop that the complement $\Omega_P\backslash\Omega_P^0$ is $\overline{\mu}^{(s)}$-negligible for all $s$.

\subsection{Formulation of the main result}

 We start by formulating the analogue of the Borodin-Olshanski Ergodic Decomposition Theorem \cite{BO}
for finite Pickrell measures.
\begin{proposition}\label{mainthm-fin} Let $s>-1$. Then $\overline{\mu}^{(s)}(\Omega_P^0)=1$ and
the  $\overline{\mu}^{(s)}$-almost sure bijection
$\omega\to \conf(\omega)$ identifies the measure $\overline{\mu}^{(s)}$ with the determinantal measure $\Prob_{J^{(s)}}$.
\end{proposition}

The main result of this paper, an explicit description for the  ergodic decomposition of infinite Pickrell measures, is given by the following
\begin{theorem} \label{mainthm}
Let $s\in {\mathbb R}$, and let $\overline{\mu}^{(s)}$
be the decomposing measure, defined by \eqref{ergdecmubar}, of the Pickrell measure ${\mu}^{(s)}$.
Then \begin{enumerate}
\item $\overline{\mu}^{(s)}(\Omega_P\backslash \Omega_P^0)=0$;

\item the $\overline{\mu}^{(s)}$-almost sure bijection $\omega\to {\rm conf}(\omega)$ identifies
$\overline{\mu}^{(s)}$ with the infinite determinantal measure
$\mathbb{B}^{(s)}$.
\end{enumerate}
\end{theorem}

\subsection{A skew-product representation of the measure ${\mathbb B}^{(s)}$}
\label{subsec-skewproduct}

With respect to the measure ${\mathbb B}^{(s)}$, almost every configuration $X$ only accumulates at zero and therefore admits a maximal
particle that we denote $x_{\max}(X)$. We are interested in the distribution of the maximal particle under the measure ${\mathbb B}^{(s)}$.
By definition, for any $R>0$, the measure ${\mathbb B}^{(s)}$ assigns finite weight to the set $\{X: x_{\max}(X)<R\}$. Furthermore, again by definition, for any
$R>0$  and $R_1, R_2\leq R$
we have the following relation:
\begin{equation*}
\frac{{\mathbb B}^{(s)}\left(\{X:x_{\max}(X)<R_1\}\right) }{{\mathbb B}^{(s)}\left(\{X:x_{\max}(X)<R_2\}\right)}=
\frac{\det\left(1-\chi_{(R_1, +\infty)}{ \Pi}^{(s, R)}\chi_{(R_1, +\infty)}\right)}{\det\left(1-\chi_{(R_2, +\infty)}{ \Pi}^{(s, R)}\chi_{(R_2, +\infty)}\right)}.
\end{equation*}

The push-forward of the measure ${\mathbb B}^{(s)}$ is a well-defined Borel sigma-finite measure on $(0, +\infty)$ for which we will use the symbol $\xi_{\max}{\mathbb B}^{(s)}$; the measure $\xi_{\max}{\mathbb B}^{(s)}$ is, of course, defined up to multiplication by a positive constant.

{\bf Question.} What is the asymptotics of the quantity $\xi_{\max}{\mathbb B}^{(s)}(0, R)$ as $R\to\infty$? as $R\to 0$?

The operator $\Pi^{(s,R)}$ admits a kernel for which we keep the same symbol; consider the function $\varphi_R(x)=\Pi^{(s,R)}(x, R)$.
By definition, $$\varphi_R(x)\in \chi_{(0, R)}H^{(s)}.$$
Let ${\overline H}^{(s,R)}$ stand for the orthogonal complement
to the one-dimensional subspace spanned by $\varphi_R(x)$ in $\chi_{(0, R)}H^{(s)}$. In other words, ${\overline H}^{(s,R)}$ is the
subspace of those functions in $\chi_{(0, R)}H^{(s)}$ that  assume value zero at the point $R$. Let ${\overline \Pi}^{(s,R)}$be the operator of orthogonal projection onto the subspace ${\overline H}^{(s,R)}$.
\begin{proposition}
We have
$$
{\mathbb B}^{(s)}=\displaystyle \int\limits_0^{\infty} \Prob_{ {\overline \Pi}^{(s,R)}} d\xi_{\max}{\mathbb B}^{(s)}(R).
$$
\end{proposition}
Proof. This immediately follows from the definition of the measure ${\mathbb B}^{(s)}$ and the characterization
of Palm measures for determinantal point processes due to Shirai and Takahashi \cite{ShirTaka1}.

\subsection{The general scheme of ergodic decomposition}
\subsubsection{Approximation}

Let $\mathfrak{F}$ be the family of $\sigma$-infinite $U(\infty)\times U(\infty)$-invariant measures $\mu$ on $\Mat(\mathbb{N},\mathbb{C})$ for which there exists $n_0$ (dependent on $\mu$) such that for all $R>0$ we have
$$\mu\biggl(\Bigl\{z\colon\max\limits_{1\leq i,j\leq n_0}\left|z_{ij}\right|<R\Bigr\}\biggr)<+\infty.$$
By definition, all Pickrell measures belong to the class $\mathfrak{F}$.

We recall the result of \cite{Buf-inferg} stating that every ergodic measure belonging to the class $\mathfrak{F}$ must be finite and that the ergodic components
of any measure in $\mathfrak{F}$ are therefore almost surely finite (the existence of the ergodic decomposition for any measure $\mu\in\mathfrak{F}$ follows from
the ergodic decomposition theorem for actions of inductively compact groups established in \cite{Buf-ergdec}).
The classification of finite ergodic measures
 now implies that for every measure $\mu\in\mathfrak{F}$ there exists a unique Borel $\sigma$-finite measure $\overline{\mu}$ on the Pickrell set $\Omega_P$ such that
\begin{equation}\label{mu-erg-dec}
\mu=\int\limits_{\Omega_P}\eta_{\omega}\,d\overline{\mu}(\omega).
\end{equation}

Our next aim is to construct, following Borodin and Olshanski \cite{BO}, a sequence of finite-dimensional approximations for the measure $\overline{\mu}$.

To a matrix $z\in\Mat(\mathbb{N},\mathbb{C})$ and a number $n\in\mathbb{N}$  assign the array
$$(\lambda_1^{(n)},\lambda_2^{(n)},\dots,\lambda_n^{(n)})$$
of eigenvalues arranged in non-increasing order of the matrix $(z^{(n)})^{\ast}z^{(n)}$, where
$$z^{(n)}=(z_{ij})_{i,j=1,\dots,n}.$$
For  $n\in\mathbb{N}$  define a map
$$
\rrad^{(n)}\colon\Mat(\mathbb{N},\mathbb{C})\to\Omega_P
$$
by the formula
\begin{equation*}
\rrad^{(n)}(z)=\left(\frac1{n^2}\tr(z^{(n)})^{\ast}z^{(n)}, \frac{\lambda_1^{(n)}}{n^2},\frac{\lambda_2^{(n)}}{n^2},\dots,\frac{\lambda_n^{(n)}}{n^2},0,0,\dots\right).
\end{equation*}
It is clear by definition that for any $n\in\mathbb{N}$, $z\in\Mat(\mathbb{N},\mathbb{C})$ we have $$\rrad^{(n)}(z)\in\Omega_P^0.$$

For any $\mu\in\mathfrak{F}$ and all sufficiently large $n\in\mathbb{N}$ the push-forwards $(\rrad^{(n)})_{\ast}\mu$ are well-defined since the unitary group is compact. We shall presently see that for any $\mu\in\mathfrak{F}$ the measures $(\rrad^{(n)})_{\ast}\mu$ approximate the
ergodic decomposition measure ${\overline \mu}$.

We start by a direct description of the map that takes a measure $\mu\in\mathfrak{F}$  to its
ergodic decomposition measure ${\overline \mu}$.

Following Borodin-Olshanski \cite{BO}, let $\Matreg(\mathbb{N},\mathbb{C})$ be the set of all matrices $z$ such that
\begin{enumerate}
\item for any $k$, there exists the limit $\displaystyle\lim\limits_{n\to\infty}\frac1{n^2}\lambda_n^{(k)}=:x_k(z)$;
\item there exists the limit $\displaystyle\lim\limits_{n\to\infty}\frac1{n^2}\tr(z^{(n)})^{\ast}z^{(n)}=:\gamma(z)$.
\end{enumerate}

Since the set of regular matrices has full measure with respect to
any finite ergodic $U(\infty)\times U(\infty)$-invariant measure, the
existence of the ergodic decomposition \eqref{mu-erg-dec} implies
$$\mu(\Mat(\mathbb{N},\mathbb{C})\big\backslash\Matreg(\mathbb{N},\mathbb{C}))=0.$$
We introduce the map
$$\rrad^{(\infty)}\colon\Matreg(\mathbb{N},\mathbb{C})\to\Omega_P$$
by the formula
$$\rrad^{(\infty)}(z)=\left(\gamma(z),x_1(z),x_2(z),\dots,x_k(z),\dots\right).$$

The Ergodic Decomposition Theorem \cite{Buf-ergdec} and the classification of ergodic unitarily-invariant measures in the form of Olshanski and Vershik imply the important equality
\begin{equation}\label{rinftymu}
(\rrad^{(\infty)})_{\ast}\mu=\overline{\mu}.
\end{equation}

{\bf Remark.} This equality has a simple analogue in the context of De Finetti's theorem: in order to obtain the ergodic decomposition of an exchangeable measure on the space of binary sequences, one just needs to consider the push-forward of the initial measure by the almost-surely defined map that to each sequence assigns the frequency
of zeros in it.

Given a complete separable metric space $Z$, we write $\Mfin(Z)$ for the space of all finite Borel measures on $Z$ endowed with the weak topology. Recall  \cite{Bogachev} that $\Mfin(Z)$ is itself a complete separable metric space:  the weak topology is induced, for instance,  by the L{\'e}vy-Prohorov metric.

We proceed to showing that the measures $(\rrad^{(n)})_{\ast}\mu$ approximate the measure $(\rrad^{(\infty)})_{\ast}\mu=\overline{\mu}$ as $n\to\infty$.
For finite measures $\mu$ the following statement is due to Borodin and Olshanski \cite{BO}.

\begin{proposition} Let $\mu$ be a finite $\sigma$-invariant measure on $\Mat(\mathbb{N},\mathbb{C})$. Then, as $n\to\infty$, we have
$$(\rrad^{(n)})_{\ast}\mu\to(\rrad^{(\infty)})_{\ast}\mu$$
weakly in $\Mfin(\Omega_P)$.
\end{proposition}

\begin{proof}

Let $f\colon\Omega_P\to\mathbb{R}$ be continuous and bounded. For any $z\in\Matreg(\mathbb{N},\mathbb{C})$, by definition, we have $\rrad^{(n)}(z)\to\rrad^{(\infty)}(z)$ as $n\to\infty$, and, consequently, also,
$$\lim\limits_{n\to\infty}f(\rrad^{(n)}(z))=f(\rrad^{(\infty)}(z)),$$
whence, by bounded convergence theorem, we have
$$\lim\limits_{n\to\infty}\int\limits_{\Mat(\mathbb{N},\mathbb{C})}f(\rrad^{(n)}(z))\,d\mu(z)= \int\limits_{\Mat(\mathbb{N},\mathbb{C})}f(\rrad^{(\infty)}(z))\,d\mu(z).$$
Changing variables, we arrive at the convergence
$$\lim\limits_{n\to\infty}\int\limits_{\Omega_P}f(\omega)\,d(\rrad^{(n)})_{\ast}\mu= \int\limits_{\Omega_P}f(\omega)\,d(\rrad^{(\infty)})_{\ast}\mu,$$
and the desired weak convergence is established.
\end{proof}

For $\sigma$-finite measures $\mu\in\mathfrak{F}$, the Borodin-Olshanski proposition is modified as follows.

\begin{lemma}\label{main-lemma} Let $\mu\in\mathfrak{F}$. There exists a positive bounded continuous function f on the Pickrell set $\Omega_P$ such that
\begin{enumerate}
\item $f\in L_1(\Omega_P,(\rrad^{(\infty)})_{\ast}\mu)$ and $f\in L_1(\Omega_P,(\rrad^{(n)})_{\ast}\mu)$ for all sufficiently large $n\in\mathbb{N}$;
\item as $n\to\infty$, we have
$$f(\rrad^{(n)})_{\ast}\mu\to f(\rrad^{(\infty)})_{\ast}\mu$$
    weakly in $\Mfin(\Omega_P)$.
\end{enumerate}

\end{lemma}
Proof of Lemma \ref{main-lemma} will be given in the sequel to this paper.

{\bf Remark.} As the above argument shows,  the explicit characterization of the ergodic decomposition of Pickrell measures given in Theorem \ref{mainthm}
does rely on the abstract result, Theorem 1 in \cite{Buf-ergdec}, that a priori guarantees the existence of the ergodic decomposition and does not by itself
give an alternative proof of the existence of the ergodic decomposition.

\subsubsection{Convergence of probability measures on the Pickrell set}
Recall that we have a natural forgetting map
$\conf:\Omega_P\to \Conf(0,+\infty)$ that
to a point $\omega=(\gamma, x)$, $x=(x_1, \dots, x_n, \dots)$, assigns the configuration
$\conf(\omega)=(x_1, \dots, x_n, \dots)$.

For $\omega\in\Omega_P$, $\omega=(\gamma, x)$, $x=(x_1, \dots, x_n, \dots)$,
$x_n=x_n(\omega)$,  set
$$
S(\omega)=\sum_{n=1}^{\infty} x_n(\omega).
$$
In other words, we set $S(\omega)=S(\conf(\omega))$, and, slightly abusing notation, keep the same symbol for the new map.
Take $\beta>0$ and consider the measures
$$\exp(-\beta S(\omega))\rrad^{(n)}(\mu^{(s)})),$$
 $n\in {\mathbb N}$.

\begin{proposition}\label{expbetaint}
For any $s\in {\mathbb R}$, $\beta>0$,  we have
$$
\exp(-\beta S(\omega))\in L_1(\Omega_P, \rrad^{(n)}(\mu^{(s)})).
$$
\end{proposition}

Introduce the probability measure
$$
\nu^{(s, n, \beta)}=\frac{\exp(-\beta S(\omega))\rrad^{(n)}(\mu^{(s)})}{\displaystyle \int\limits_{\Omega_P} \exp(-\beta S(\omega))d\rrad^{(n)}(\mu^{(s)})}.
$$

Now go back to the determinantal measure $\Prob_{\Pi^{(s, \beta)}}$ on the space $\Conf((0, +\infty))$ (cf. \eqref{pisb}) and let the measure
$\nu^{(s, \beta)}$ on $\Omega_P$ be defined by the requirements
\begin{enumerate}
\item $\nu^{(s, \beta)}(\Omega_P\setminus \Omega_P^0)=0$;
\item $\conf_*\nu^{(s, \beta)}=\Prob_{\Pi^{(s, \beta)}}.$
\end{enumerate}
The key r{\^o}le in the proof of Theorem \ref{mainthm} is played by
\begin{proposition}\label{weak-pick}
For any $\beta>0$, $s\in {\mathbb R}$, as $n\to\infty$ we have
 $$\nu^{(s,n,\beta)}\to \nu^{(s, \beta)}$$
weakly in the space $\Mfin(\Omega_P)$.
\end{proposition}
Proposition \ref{weak-pick} will be proved in Section III.3 and in Section III.4, using Proposition \ref{weak-pick}, combined with Lemma \ref{main-lemma}, we will conclude the proof of the main result, Theorem \ref{mainthm}.

To establish weak convergence of the measures $\nu^{(s,n,\beta)}$, we first study scaling limits of the radial parts of finite-dimensional
projections of infinite Pickrell measures.

\subsection{The radial part of the Pickrell measure}

Following Pickrell,  to a matrix $z\in \Mat(n,\mathbb{C})$ assign the collection
$\left(\lambda_1(z),\dots,\lambda_n(z)\right)$ of the
eigenvalues of the matrix $z^*z$ arranged in non-increasing order.
Introduce a
map
$$\mathfrak {rad}_n: \Mat(n, \mathbb C)\to {\mathbb R}_+^n
$$
by the formula
\begin{equation}\label{radn-def}
\mathfrak {rad}_n: z\to\left(\lambda_1(z),\dots,\lambda_n(z)\right).
\end{equation}
The map \eqref{radn-def} naturally extends to a map defined on $\Mat(\mathbb N, \mathbb C)$ for which we keep the same symbol:
in other words, the map $\mathfrak {rad}_n$ assigns to an infinite matrix the array of squares of the singular values of its $n\times n$-corner.

The {\it radial part}
of the Pickrell
measure $\mu_n^{(s)}$
 is now defined as the
push-forward of the  measure
$\mu_n^{(s)}$ under the
map
$\mathfrak {rad}_n$.
Note that, since finite-dimensional unitary groups are compact, and, by definition, for any $s$ and all sufficiently large $n$,  the measure $\mu_n^{(s)}$  assigns finite weight to compact sets,  the pushforward is well-defined, for  sufficiently large $n$, even if the measure $\mu^{(s)}$ is infinite.

Slightly abusing notation, we write $dz$ for the Lebesgue measure
$\Mat(n,\mathbb{C})$ and $d\la$ for the Lebesgue measure on ${\mathbb R}_+^n$.

For the push-forward of the Lebesgue measure $\Leb^{(n)}=dz$ under the map
$\mathfrak{ rad}_n$ we now have
$$
(\mathfrak {rad}_n)_*(dz)=\const (n) \cdot \prod_{i<j} (\la_i-\la_j)^2 d\la,
$$
where $\const(n)$ is a positive constant depending only on $n$.

 The radial part of the measure $\mu_n^{(s)}$ now takes the form:
\begin{equation*}
(\mathfrak {rad}_n)_*\mu_n^{(s)} =
\const (n,s) \cdot \prod_{i<j} (\la_i-\la_j)^2 \cdot \frac1{(1+\la_i)^{2n+s}}d\la,
\end{equation*}
where $\const(n,s)$ for a positive constant depending on $n$ and $s$ (the constant may change from one formula to another).

 Following Pickrell, introduce new variables $u_1, \dots, u_n$ by the formula
\begin{equation}\label{ux-chg}
u_i = \frac{\la_i-1}{\la_i+1}.
\end{equation}
\begin{proposition}\label{rad-jac}
In the  coordinates  \eqref{ux-chg} the radial part  $(\mathfrak {rad}_n)_*\mu_n^{(s)}$
of the measure $\mu_n^{(s)}$
is defined on the cube $[-1,1]^n$  by the formula
\begin{equation}\label{rad-u-coord}
(\mathfrak {rad}_n)_*\mu_n^{(s)}=
\const (n,s) \cdot \prod_{i<j} (u_i-u_j)^2\cdot \prod_{i=1}^n (1-u_i)^s \, du_i.
\end{equation}
\end{proposition}
 In the case $s>-1$, the constant $\const (n,s)$ can be chosen in such a way that
the right-hand side be a probability measure; in the case $s\leq -1$, there is no canonical normalization, the  left hand side is defined up to proportionality, and a positive constant can be chosen arbitrarily.

For $s>-1$, Proposition \ref{rad-jac} yields a determinantal representation for the radial part of the Pickrell measure: namely, the radial part is identified with the Jacobi orthogonal polynomial ensemble in the coordinates \eqref{ux-chg}. Passing to the scaling limit, one obtains the Bessel point process
(subject to the  change of variable $y=4/x$).

Similarly, it will develop that for $s\leq -1$, the scaling limit of the measures \eqref{rad-u-coord} is precisely the modified infinite Bessel point process introduced above.
Furthermore, if one multiplies the measures \eqref{rad-u-coord}  by the density $\exp(-\beta S(X)/n^2)$, then the resulting measures are finite and determinantal,
and their weak limit, after appropriate scaling,  is precisely the determinantal measure $\Prob_{\Pi^{(s, \beta)}}$ of \eqref{pisb}. This weak convergence  is a key step in the proof of Proposition \ref{weak-pick}.

The study of the case $s\leq -1$ thus  requires a new object: infinite determinantal measures on spaces of configurations. In the next section, we proceed to the general construction and description of the properties of infinite determinantal measures.

{\bf {Acknowledgements.}}
Grigori Olshanski posed the problem to me, and I am greatly indebted to him.
I am deeply grateful to Alexei M. Borodin, Yanqi  Qiu, Klaus Schmidt and Maria V. Shcherbina for useful discussions.

The author is supported by A*MIDEX project (No. ANR-11-IDEX-0001-02), financed by Programme ``Investissements d'Avenir'' of the Government of
the French Republic managed by the French National Research Agency (ANR).
It is also supported by the Grant MD 2859.2014.1 of the President of the Russian Federation, by the Russian
Foundation of Basic Research (grant 13-01-12449 ofi-m) and by the  Russian Academic Excellence Project
'5-100'.

Part of this work was done at the Institut Henri Poincar{\'e} in Paris,
Institut Mittag-Leffler of the Royal Swedish Academy of Sciences in Djursholm,
and the Institute of Mathematics ``Guido Castelnuovo'' of the University of
Rome ``La Sapienza''. I am deeply grateful to these institutions for their warm hospitality.

\section{Construction and properties of infinite determinantal measures}

\subsection{Preliminary remarks  on sigma-finite measures}
Let $Y$ be a Borel space, and consider a representation
$$
Y=\bigcup\limits_{n=1}^{\infty} Y_n
$$
of $Y$ as a countable union of an increasing sequence of subsets $Y_n$, $Y_n\subset Y_{n+1}$.
As before, given a measure $\mu$ on $Y$ and a subset $Y'\subset Y$, we write $\mu|_{Y'}$ for the restriction
of $\mu$ onto $Y'$.
Assume that for every $n$ we are given a probability measure $\Prob_n$ on $Y_n$.
The following proposition is clear.
\begin{proposition}\label{sigmf-gen}
A sigma-finite measure ${\mathbb B}$ on $Y$ such that
\begin{equation}\label{ind-byn}
\frac{{\mathbb B}|_{Y_n}}{{\mathbb B}(Y_n)}=\Prob_n
\end{equation}
exists if and only if for any $N, n$, $N>n$, we have
$$
\frac{\Prob_N|_{Y_n}}{\Prob_N(Y_n)}=\Prob_n.
$$
The condition \eqref{ind-byn} determines the measure ${\mathbb B}$ uniquely up to multiplication by a constant.
\end{proposition}

\begin{corollary} If ${\mathbb B}_1$, ${\mathbb B}_2$ are two  sigma-finite measures on $Y$ such that for all $n\in {\mathbb N}$ we have
$$
0<{\mathbb B}_1(Y_n)<+\infty, 0<{\mathbb B}_2(Y_n)<+\infty,
$$
and
$$
\frac{{\mathbb B}_1|_{Y_n}}{{\mathbb B}_1(Y_n)}=\frac{{\mathbb B}_2|_{Y_n}}{{\mathbb B}_2(Y_n)},
$$
then there exists a positive constant $C>0$ such that ${\mathbb B}_1=C{\mathbb B}_2$.
\end{corollary}

\subsection{The unique extension property}

\subsubsection{Extension from a subset}

Let $E$ be a standard Borel space,  let $\mu$ be a sigma-finite measure on $E$, let $L$ be a closed subspace of $L_2(E, \mu)$, let $\Pi$ be the operator of orthogonal projection onto $L$,
and let $E_0\subset E$ be a Borel subset. We shall say that the subspace $L$ has the {\it unique extension property} from $E_0$
if a function $\varphi\in L$ satisfying $\chi_{E_0}\varphi=0$ must be the zero function and the subspace $\chi_{E_0}L$ is closed.
 In general, if  a function $\varphi\in L$ satisfying $\chi_{E_0}\varphi=0$ must be the zero function, then the restricted subspace $\chi_{E_0}L$ still  need not be closed: nonetheless, we have the following clear corollary of the open mapping theorem.
\begin{proposition}\label{closed-sub}
Assume that the closed subspace $L$  is such that
a function $\varphi\in L$ satisfying $\chi_{E_0}\varphi=0$ must be the zero function.
The subspace $\chi_{E_0}L$ is closed if and only if there exists $\varepsilon>0$ such that
for any $\varphi\in L$ we have
\begin{equation}\label{eps-closed}
||\chi_{E\setminus E_0}\varphi||\leq (1-\varepsilon)||\varphi||,
\end{equation}
in which case the natural restriction map $\varphi\to \chi_{E_0}\varphi$ is an isomorphism of Hilbert spaces.
If the operator $\chi_{E\setminus E_0}\Pi$ is compact, then the condition \eqref{eps-closed} holds.
\end{proposition}
{\bf Remark.} In particular, the condition \eqref{eps-closed}  a fortiori holds if the operator $\chi_{E\setminus E_0}\Pi$  is Hilbert-Schmidt or, equivalently, if the operator $\chi_{E\setminus E_0}\Pi\chi_{E\setminus E_0}$
belongs to the trace class.

The following  corollaries are immediate.
\begin{corollary}
Let $g$ be a bounded nonegative Borel  function on $E$ such that
\begin{equation}\label{inf-g-pos}
\inf\limits_{x\in E_0} g(x)>0.
\end{equation}
If \eqref{eps-closed} holds then the subspace $\sqrt{g} L$  is closed in $L_2(E, \mu)$.
\end{corollary}
{\bf Remark.} The apparently superfluous square root is put here to keep notation consistent with the remainder of the paper.
\begin{corollary}
Under the assumptions of Proposition \ref{closed-sub}, if  \eqref{eps-closed} holds and a Borel function $g: E\to [0,1]$ satisfies \eqref{inf-g-pos}, then
the operator $\Pi^g$  of orthogonal projection onto the subspace $\sqrt{g}L$ is given by the formula
\begin{equation}\label{def-Pi-g}
\Pi^g=\sqrt{g}\Pi(1+(g-1)\Pi)^{-1}\sqrt{g}=\sqrt{g}\Pi(1+(g-1)\Pi)^{-1}\Pi\sqrt{g}.
\end{equation}
In particular, the operator $\Pi^{E_0}$ of orthogonal projection onto the subspace $\chi_{E_0}L$ has the form
\begin{equation}\label{proj-enol}
\Pi^{E_0}=\chi_{E_0}\Pi(1-\chi_{E\setminus E_0}\Pi)^{-1}\chi_{E_0}=\chi_{E_0}\Pi(1-\chi_{E\setminus E_0}\Pi)^{-1}\Pi\chi_{E_0}.
\end{equation}
\end{corollary}
\begin{corollary}\label{hs-major}
Under the assumptions of Proposition \ref{closed-sub}, if  \eqref{eps-closed} holds, then, for any subset $Y\subset E_0$, once
the operator $\chi_Y\Pi^{E_0}\chi_Y$ belongs to the trace class,  it follows that so does  the operator $\chi_Y\Pi\chi_Y$, and we have
$$
\tr \chi_Y\Pi^{E_0}\chi_Y\geq \tr \chi_Y\Pi\chi_Y
$$
\end{corollary}
Indeed, from \eqref{proj-enol} it is clear that  if the operator $\chi_Y\Pi^{E_0}$ is Hilbert-Schmidt, then the operator $\chi_Y\Pi$ is also Hilbert-Schmidt.
The inequality between traces is also immediate from  \eqref{proj-enol}.

\subsubsection{Examples: the Bessel kernel and the modified Bessel kernel}
\begin{proposition}
\begin{enumerate}
\item For any $\varepsilon>0$, the operator ${\tilde J}_s$ has the unique extension property from the subset
$(\varepsilon, +\infty)$;
\item For any $R>0$, the operator ${J}^{(s)}$ has the unique extension property from the subset
$(0, R)$.
\end{enumerate}
\end{proposition}

Proof. The first statement is an immediate corollary of the uncertainty principle for the Hankel transform: a function and its Hankel transform cannot both have support of finite measure \cite{ghobber1}, \cite{ghobber2} (note here that the uncertainty principle is only formulated for $s>-1/2$ in \cite{ghobber1} but the more general uncertainty principle of \cite{ghobber2} is directly applicable also to the case $s\in[-1,1/2]$) and the following estimate, which, by definition, is clearly valid for any $R>0$:
$$
\int_0^R {\tilde J}_s(y,y)dy<+\infty.
$$
The second statement follows from the first by the change of variable $y=4/x$.
The proposition is proved completely.

\subsection{Inductively determinantal measures}
Let $E$ be a locally compact complete metric space, and let $\Conf(E)$ be the space of configurations on $E$ endowed with the natural Borel
structure (see, e.g., \cite {Lenard},~\cite{Soshnikov} and the Subsection~\ref{subsec-space-conf}).

Given a Borel subset $E'\subset E$, we let $\Conf(E, E')$ be the subspace of
configurations all whose particles lie in $E'$.

Given a measure $\mathbb{B}$ on a set $X$ and a measurable subset $Y\subset X$ such
that $0<\mathbb{B}(Y)<+\infty$, we let $\mathbb{B}\left|_Y\right.$ stand for the
restriction of the measure $\mathbb{B}$ onto the subset $Y$.

Let $\mu$ be a $\sigma$-finite Borel measure on $E$.

We let $E_0\subset E$ be a Borel subset and assume that for any bounded
Borel subset $B\subset E\backslash E_0$ we are given a closed subspace
$L^{E_0\cup B}\subset L_2(E,\mu)$ such that the corresponding projection operator
$\Pi^{E_0\cup B}$ belongs to the space
$\scrI_{1,\loc}(E,\mu)$. We furthermore make the following

\begin{assumption}\label{ind-det}
\begin{enumerate}
\item $\displaystyle\left\|\chi_B\Pi^{E_0\cup B}\right\|<1\,$,\;
$\displaystyle\chi_B\Pi^{E_0\cup B}\chi_B\in\scrI_1(E,\mu)$
\item for any subsets $B^{(1)}\subset B^{(2)}\subset E\backslash E_0$,
we have $$\displaystyle
\chi_{E_0\cup B^{(1)}}L^{E_0\cup B^{(2)}}=L^{E_0\cup B^{(1)}}.$$
\end{enumerate}
\end{assumption}
\begin{proposition}\label{induc-det-def}
Under these assumptions, there exists a $\sigma$-finite measure
$\mathbb{B}$ on $\Conf(E)$ such that
\begin{enumerate}
\item for $\mathbb{B}$-almost every configuration, only finitely many of
its particles may lie in $E\backslash E_0\,$;
\item for any bounded Borel subset $B\subset E\backslash E_0$, we have
$$0<\mathbb{B}\bigl(\Conf(E;\,E_0\cup B)\bigr)<+\infty\;\;\text{and}$$
$$\frac{\mathbb{B}\left|_{\Conf(E;E_0\cup B)}\right.}
{\mathbb{B}\bigl(\Conf(E;E_0\cup B)\bigr)}=
\Prob_{\Pi^{E_0\cup B}}\;.$$
\end{enumerate}
\end{proposition}
Such a measure will be called an {\it inductively determinantal measure}.

Proposition \ref{induc-det-def} is immediate from Proposition \ref{sigmf-gen} combined with
 Proposition \ref{pr1-bis} and Corollary \ref{indsubset}. Note that conditions 1 and 2 define our measure uniquely up to
multiplication by a constant.

We now give a sufficient condition for an inductively determinantal measure
to be an actual finite determinantal measure.

\begin{proposition} Consider a family of projections
$\Pi^{E_0\cup B}$ satisfying the Assumption \ref{ind-det} and the corresponding inductively determinantal measure
$\mathbb{B}$. If there exists $R>0$, $\varepsilon>0$ such that for all
bounded Borel subset $B\subset E\backslash E_0$ we have
\begin{enumerate}
\item $\left\|\chi_B\Pi^{E_0\cup B}\right\|<1-\varepsilon$;
\item $\tr\chi_B\Pi^{E_0\cup B}\chi_B<R$.
\end{enumerate}
then there exists a projection operator $\Pi\in\scrI_{1,\loc}(E,\mu)$
onto a closed subspace $L\subset L_2(E,\mu)$ such that
\begin{enumerate}
\item $L^{E_0\cup B}=\chi_{E_0\cup B}L\,$ for all $B$;
\item $\chi_{E\backslash E_0}\Pi\chi_{E\backslash
E_0}\in\scrI_1(E,\mu)$;
\item the measures $\mathbb{B}$ and $\Prob_{\Pi}$ coincide up to
multiplication by a constant.
\end{enumerate}
\end{proposition}
Proof. By our assumptions, for every bounded  Borel subset $B\subset E\backslash E_0$
we are given a closed subspace $L^{E_0\cup B}$, the range of the operator $\Pi^{E_0\cup B}$, which has the property of unique extension from $E_0$. The uniform estimate on the norms of the operators $\chi_B\Pi^{E_0\cup B}$ implies the existence of a closed subspace $L$ such that $L^{E_0\cup B}=\chi_{E_0\cup B}L$.
Now, by our assumptions, the projection operator
$\Pi^{E_0\cup B}$ belongs to the space
$\scrI_{1,\loc}(E,\mu)$, whence,  for any bounded subset $Y\subset E$, we have
$$
\chi_Y\Pi^{E_0\cup Y}\chi_Y\in \scrI_{1}(E,\mu),
$$
whence, by Corollary \ref{hs-major} applied to the subset $E_0\cup Y$, it follows that
$$
\chi_Y\Pi\chi_Y\in \scrI_{1}(E,\mu).
$$
It follows that the operator $\Pi$ of orthogonal projection on $L$ is locally of trace class and therefore induces a unique
determinantal probability measure $\Prob_{\Pi}$ on $\Conf(E)$.
Applying Corollary \ref{hs-major} again, we have
$$
\tr \chi_{E\setminus E_0}\Pi \chi_{E\setminus E_0}\leq R,
$$
and the proposition is proved completely.

We now give  sufficient conditions for the measure
$\mathbb{B}$ to be infinite.

\begin{proposition} \label{suf-inf} Make either of the two assumptions:
\begin{enumerate}
\item for any $\varepsilon>0$, there exists a bounded
Borel subset $B\subset E\backslash E_0$ such that
$$
||\chi_B\Pi^{E_0\cup B}||>1-\varepsilon
$$
\item
 for any $R>0$, there exists a bounded
Borel subset $B\subset E\backslash E_0$ such that
$$
\tr\chi_B\Pi^{E_0\cup B}\chi_B>R\,.
$$
\end{enumerate}
Then the measure $\mathbb{B}$ is infinite.
\end{proposition}
{Proof}.
Recall that we have
\begin{multline*}
\frac{\mathbb{B}\bigl(\Conf(E;E_0)\bigr)}{\mathbb{B}\bigl(\Conf(E;E_0\cup B)\bigr)}=
\Prob_{\Pi^{E_0\cup B}}\bigl(\Conf(E;E_0)\bigr)=\\
=\det(1-\chi_B\Pi^{E_0\cup B}\chi_B).
\end{multline*}
Under the first assumption, it is immediate that the top eigenvalue of the self-adjoint trace-class operator $\chi_B\Pi^{E_0\cup B}\chi_B$ exceeds $1-\varepsilon$, whence
$$
\det\left(1-\chi_B\Pi^{E_0\cup B}\chi_B\right)\leq \varepsilon.
$$
Under the second assumption, write
\begin{equation*}
\det\left(1-\chi_B\Pi^{E_0\cup B}\chi_B\right)\leq
\exp\left(-\tr\chi_B\Pi^{E_0\cup B}\chi_B\right)\leq \exp(-R).
\end{equation*}
In both cases, the ratio
$$
\frac{\mathbb{B}\bigl(\Conf(E;E_0)\bigr)}{\mathbb{B}\bigl(\Conf(E;\,E_0\cup B)\bigr)}
$$
can be made arbitrary small by an appropriate choice of $B$, which implies that the measure $\mathbb{B}$ is infinite.
The proposition is proved.

\subsection{General construction of infinite determinantal measures}

By the Macch{\`i}-Soshnikov Theorem, under some additional assumptions, a determinantal measure can be assigned to an operator of orthogonal projection, or, in other words,  to a closed subspace of $L_2(E, \mu)$. In a similar way, an infinite determinantal measure will be assigned to a subspace $H$  of {\it locally} square-integrable functions.

Recall that $L_{2,  \loc}(E, \mu)$ is the space of all measurable functions $f:E\to {\mathbb C}$
such that for any bounded subset $B\subset E$ we have
\begin{equation}
\label{bintef}
\int\limits_B |f|^2d\mu<+\infty.
\end{equation}

Choosing an exhausting family $B_n$ of bounded sets (for instance,
balls with fixed centre and of radius tending to infinity) and using \eqref{bintef} with
$B=B_n$, we endow the space $L_{2, \loc}(E, \mu)$ with a countable
family of seminorms which turns it into a complete separable metric
space; the topology thus defined does not, of course, depend on the
specific choice of the exhausting family.

Let $H\subset L_{2,\loc}(E,\mu)$ be a linear subspace. If
$E'\subset E$ is a Borel subset such that $\chi_{E'}H$ is a closed subspace of
$L_2(E,\mu)$, then we denote by $\Pi^{E'}$ the operator of
orthogonal projection onto the subspace $\chi_{E'}H\subset
L_2(E,\mu).$ We now fix a Borel subset $E_0\subset E$;  informally, $E_0$ is the set where the particles accumulate. We impose
the following assumption on $E_0$ and $H$.
\begin{assumption}\label{he}
\begin{enumerate}
\item For any bounded Borel set $B\subset E$, the space  $\chi_{E_0\cup B} H$ is a closed subspace of $L_2(E,\mu)$;
 \item For any bounded Borel set $B\subset E\setminus E_0$, we have
\begin{equation*}
\Pi^{E_0\cup B}\in \scrI_{1,\loc}(E,\mu),\quad \chi_B\Pi^{E_0\cup B}\chi_B\in\scrI_{1}(E,\mu);
\end{equation*}
\item If $\varphi\in H$ satisfies $\chi_{E_0}\varphi =0$, then $\varphi=0.$
\end{enumerate}
\end{assumption}

If a subspace $H$ and the subset $E_0$ have the property that any $\varphi\in H$ satisfying $\chi_{E_0}\varphi =0$ must be the zero function,
then we shall say that $H$ has the {\it property of unique extension} from $E_0$.

\begin{theorem}\label{infdet-he}
Let $E$ be a locally compact complete metric space, and let $\mu$ be
a $\sigma$-finite Borel measure on $E$. If a subspace
$H\subset L_{2,\loc}(E,\mu)$ and a Borel subset
$E_0\subset E$ satisfy Assumption \ref{he}, then there exists a
$\sigma$-finite Borel measure $\mathbb{B}$ on $\Conf(E)$
such that
\begin{enumerate}
\item ${\mathbb B}$-almost every configuration has at most finitely
many particles outside of $E_0$;
\item for any bounded Borel (possibly empty) subset
$B\subset E\setminus E_0$ we have $0<\mathbb{B}(\Conf(E;E_0\cup B))<+\infty$
and
$$
\frac{\mathbb{B}|_{\Conf(E;E_0\cup B)}}{\mathbb{B}(\Conf(E;E_0\cup B))}=
\Prob_{\Pi^{E_0\cup B}}.
$$
\end{enumerate}
The requirements (1) and (2) determine the measure $\mathbb{B}$
uniquely up to multiplication by a positive constant.
\end{theorem}
We denote ${\mathbf B}(H,E_0)$ the one-dimensional cone of nonzero infinite determinantal measures induced by
$H$ and $E_0$, and, slightly abusing notation, we write ${\mathbb B}={\mathbb B}(H, E_0)$
for a representative of the cone.

{\bf {Remark.}} If $B$ is a bounded set, then, by definition, we have
$${\mathbf B}(H,E_0)={\mathbf B}(H,E_0\cup B).$$

{\bf {Remark.}}  If $E'\subset E$ is a Borel subset such that $\chi_{E_0\cup E'}$ is
a closed subspace in $L_2(E, \mu)$ and the operator $\Pi^{E_0\cup E'}$ of orthogonal projection onto the subspace
$\chi_{E_0\cup E'}H$ satisfies
\begin{equation*}
\Pi^{E_0\cup E'}\in \scrI_{1,\loc}(E,\mu),\quad \chi_{E'}\Pi^{E_0\cup E'}\chi_{E'}\in\scrI_{1}(E,\mu),
\end{equation*}
then, exhausting $E'$ by bounded sets, from Theorem \ref{infdet-he} one easily obtains
$0<\mathbb{B}(\Conf(E;E_0\cup E'))<+\infty$
and
$$
\frac{\mathbb{B}|_{\Conf(E;E_0\cup E')}}
{\mathbb{B}({\Conf(E;E_0\cup E')})}=
\Prob_{\Pi^{E_0\cup E'}}.
$$

\subsection{Change of variables for infinite determinantal measures}

Let $F:E\to E$ be a homeomorphism.  The  homeomorphism $F$ induces a homeomorphism of the space $\Conf(E)$, for which, slightly abusing notation, we keep the same symbol: given $X\in\Conf(E)$, the particles of the configuration
$F(X)$ have the form $F(x)$ over all $x\in X$.

Assume now that the measures $F_*\mu$ and $\mu$ are equivalent, and let
${\mathbb B}={\mathbb B}(H, E_0)$ be an infinite determinantal measure.
Introduce the subspace
$$
F^*H=\biggl\{\varphi(F(x))\cdot \sqrt{\frac{dF_*\mu}{d\mu}}, \varphi\in H\biggr\}.
$$

From the definitions we now clearly have the following
\begin{proposition}
The push-forward of the infinite determinantal measure ${\mathbb B}={\mathbb B}(H, E_0)$ has the form
$$
F_*{\mathbb B}={\mathbb B}(F^*H, F(E_0)).
$$
\end{proposition}

\subsection{Example: infinite orthogonal polynomial ensembles}

Let $\rho$ be a nonnegative  function on ${\mathbb R}$ not identically equal to zero.
 Take $N\in\mathbb{N}$ and
endow the set ${\mathbb R}^N$ with  the measure
\begin{equation}\label{ore}
\prod\limits_{1\leq i,j\leq N}(x_i-x_j)^2\prod
\limits_{i=1}^N\rho(x_i)dx_i.
\end{equation}
If for $k=0, \dots, 2N-2$ we have
$$
\int_{-\infty}^{+\infty} x^k \rho(x)dx<+\infty,
$$
then the measure \eqref{ore} has finite mass and, after normalization, yields a determinantal point process on $\Conf({\mathbb R})$.

Given a finite family of functions $f_1, \dots, f_N$ on the real line, let $\Span(f_1, \dots, f_N)$ stand for the vector space these functions span.
For a general function $\rho$, introduce the subspace
$H(\rho)\subset L_{2, \loc}({\mathbb R},\Leb)$
by the formula
$$
H(\rho)=\Span\left(\sqrt{\rho(x)}, x\sqrt{\rho(x)}, \dots, x^{N-1}\sqrt{\rho(x)}\right).
$$
The measure \eqref{ore} is an infinite determinantal measure, as is shown by the following immediate
\begin{proposition}
Let $\rho$ be a non-negative continuous function on ${\mathbb R}$, and let $(a,b)\subset {\mathbb R}$
be a nonempty interval such that the function $\rho$ is positive in restriction to $(a,b)$.
Then the measure \eqref{ore} is an infinite determinantal measure of the form ${\mathbb B}(H(\rho), (a,b))$.
\end{proposition}

\subsection{Multiplicative functionals of infinite determinantal measures}

Our next aim is to show that, under some additional assumptions, an infinite determinantal measure can be represented as a product of a finite determinantal measure and a multiplicative functional.

\begin{proposition}\label{mult-he1}
Let a subspace $H\subset L_{2, \loc}(E, \mu)$ and a Borel subset $E_0$ induce an infinite determinantal measure
$\mathbb{B}=\mathbb{B}\left(H,E_0\right)$.
Let $g\colon E\to(0,1]$ be a positive Borel function such that $\sqrt{g}H$
is a closed subspace in $L_2(E,\mu)$, and let $\Pi^g$ be the corresponding
projection operator.
Assume additionally
\begin{enumerate}
\item $\sqrt{1-g}\Pi^{E_0}\sqrt{1-g}\in\scrI_1(E,\mu)$\,;
\item $\chi_{E\backslash E_0}\Pi^g\chi_{E\backslash
E_0}\in\scrI_1(E,\mu)$\,;
\item $\Pi^g\in\scrI_{1,\loc}(E,\mu)$
\end{enumerate}

Then the multiplicative functional $\Psi_g$ is $\mathbb{B}$-almost surely
positive, $\mathbb{B}$-integrable, and we have

$$\frac{\Psi_g\mathbb{B}}{\displaystyle \int\limits_{\Conf(E)}\Psi_g\,d\mathbb{B}}=\Prob_{\Pi^g}\;.$$
\end{proposition}

Before starting the proof, we prove some auxiliary propositions.

First, we note a simple corollary of unique extension property.

\begin{proposition}\label{unique-ext} \,. \; Let $H\subset
L_{2,\loc}(E,\mu)$ have the property of unique extension from
$E_0$, and let $\psi\in L_{2,\loc}(E,\mu)$ be such that
$\chi_{E_0\cup B}\psi\in\chi_{E_0\cup B}H$ for any bounded Borel set $B\subset E\backslash
E_0$. Then $\psi\in H$.
\end{proposition}

{Proof}\,. \; Indeed, for any $B$ there exists $\psi_B\in
L_{2,\loc}(E,\mu)$ such that $\chi_{E_0\cup B}\psi_B=\chi_{E_0\cup B}\psi$. Take two bounded Borel sets $B_1$ and
$B_2$ and note that $
\chi_{E_0}\psi_{B_1}=\chi_{E_0}\psi_{B_2}=\chi_{E_0}\psi$\,, whence, by the
unique extension property, $\psi_{B_1}=\psi_{B_2}$. Thus all the functions
$\psi_B$ coincide and also coincide with $\psi$, which, consequently,
belongs to $H$.

Our next proposition gives a sufficient condition for a subspace of locally
square-integrable functions to be a closed subspace in $L_2$.
\begin{proposition}  Let $L\subset L_{2,\loc}(E,\mu)$ be
a subspace such that
\begin{enumerate}
\item for any bounded Borel $B\subset E\backslash E_0$ the space
$\chi_{E_0\cup B}L$ is a closed subspace of $L_2(E,\mu)$;
\item the natural restriction map $\chi_{E_0\cup B}L\to\chi_{E_0}L$ is
an isomorphism of Hilbert spaces, and the norm of its inverse is bounded
above by a positive constant independent of $B$.
\end{enumerate}

Then $L$ is a closed subspace of $L_2(E,\mu)$, and the natural restriction
map $L\to\chi_{E_0}L$ is an isomorphism of Hilbert spaces.
\end{proposition}

Proof. If $L$ contained a function with non-integrable
square, then for an appropriately chosen $B$ the inverse of the
restriction isomorphism $\chi_{E_0\cup B}L\to\chi_{E_0}L$ would have an
arbitrarily large norm. That $L$ is closed follows from the unique
extension property and Proposition \ref{unique-ext}.

We now proceed with the proof of Proposition \ref{mult-he1}.

First we check that for any bounded Borel $B\subset
E\backslash E_0$ we have
\begin{equation*}
\sqrt{1-g}\Pi^{E_0\cup B}\sqrt{1-g}\in\scrI_1(E,\mu).
\end{equation*}
Indeed, the definition of an infinite determinantal measure implies
$$\chi_B\Pi^{E_0\cup B}\in\scrI_2(E,\mu),$$ whence, a fortiori, we
have $$\sqrt{1-g}\chi_B\Pi^{E_0\cup B}\in\scrI_2(E,\mu).$$
Now
recall that $$\Pi^{E_0}=\chi_{E_0}\Pi^{E_0\cup B}\left(1-\chi_B\Pi^{E_0\cup B}\right)^{-1}\Pi^{E_0\cup B}\chi_{E_0}.$$
The relation
$$
\sqrt{1-g}\Pi^{E_0}\sqrt{1-g}\in\scrI_1(E,\mu)
$$
therefore implies
$$
\sqrt{1-g}\chi_{E_0}\Pi^{E_0\cup B}\chi_{E_0}\sqrt{1-g}\in\scrI_1(E,\mu),
$$
or, equivalently,
$$
\sqrt{1-g}\chi_{E_0}\Pi^{E_0\cup B}\in\scrI_2(E,\mu).
$$
We conclude that
$$
\sqrt{1-g}\Pi^{E_0\cup B}\in\scrI_2(E,\mu),
$$
or, equivalently, that
$$
\sqrt{1-g}\Pi^{E_0\cup B}\sqrt{1-g}\in\scrI_1(E,\mu)
$$
as desired.

We next check that the subspace $\sqrt{g}H\chi_{E_0\cup B}$ is closed in
$L_2(E,\mu)$. But this is immediate from closedness of the subspace
$\sqrt{g}H$, the unique extension property from the subset $E_0$, which the
subspace $\sqrt{g}H$ has, since so does $H$, and our assumption
$$\chi_{E\backslash E_0}\Pi^g\chi_{E\backslash E_0}\in \scrI_1(E,\mu).$$

We now let $\Pi^{g\chi_{E_0\cup B}}$ be the operator of orthogonal
projection onto the subspace $\sqrt{g}H\chi_{E_0\cup B}$.

It follows from the above that for any bounded Borel set $B\subset E\backslash E_0$
the multiplicative functional $\Psi_g$ is $\Prob_{\Pi^{E_0\cup B}}$-almost surely positive and, furthermore, that we have

$$
\frac{\Psi_g\Prob_{\Pi^{E_0\cup B}}}{\displaystyle \int\Psi_g\,d\Prob_{\Pi^{E_0\cup B}}}=
\Prob_{\Pi^{g\chi_{E_0\cup B}}}\;,
$$
where $\Pi^{g\chi_{E_0\cup B}}$ is the operator of orthogonal projection
onto the closed subspace $\sqrt{g}\chi_{E_0\cup B}H$.

It follows now that for any bounded Borel $B\subset E\backslash E_0$ we have
\begin{equation}\label{restr-mult}
\frac{\Psi_{g\chi_{E_0\cup B}}\mathbb{B}}{\displaystyle \int\Psi_{g\chi_{E_0\cup B}}\,d\mathbb{B}}=
\Prob_{\Pi^{g\chi_{E_0\cup B}}}\;.
\end{equation}
It remains to note that \eqref{restr-mult} immediately implies the statement
of Proposition \ref{mult-he1}, whose proof is thus complete.

\subsection{Infinite determinantal measures obtained as finite-rank perturbations of  determinantal probability measures}

\subsubsection{Construction of finite-rank perturbations}
We now consider infinite determinantal measures induced by subspaces $H$ obtained by adding a
finite-dimensional subspace $V$ to a closed subspace $L\subset L_2(E, \mu)$.

Let, therefore,  $Q\in\scrI_{1,\loc}(E,\mu)$ be the operator of orthogonal projection onto a closed subspace $L\subset L_2(E, \mu)$, let $V$ be a finite-dimensional subspace of $L_{2, \loc}(E, \mu)$ such that $V\cap L_2(E, \mu)=0$, and set $H=L+V$.  Let  $E_0\subset E$ be a Borel subset.
We shall need the following assumption on $L, V$ and $E_0$.
\begin{assumption}\label{lve}
\begin{enumerate}
\item $\chi_{E\setminus E_0}Q\chi_{E\setminus E_0}\in\scrI_1(E,\mu)$;
\item $\chi_{E_0}V\subset L_2(E,\mu)$;
\item if $\varphi\in V$ satisfies $\chi_{E_0}\varphi\in \chi_{E_0}L$, then $\varphi=0$;
\item
if $\varphi\in L$ satisfies $\chi_{E_0}\varphi=0$, then $\varphi=0$.
\end{enumerate}
\end{assumption}

\begin{proposition} If  $L$, $V$ and  $E_0$ satisfy Assumption \ref{lve} then
the subspace  $H=L+V$ and  $E_0$ satisfy Assumption \ref{he}.
\end{proposition}

In particular, for any bounded Borel subset $B$,
the subspace  $\chi_{E_0\cup B}L$ is closed, as  one sees by taking $E'=E_0\cup B$ in the following clear
\begin{proposition}\label{closedsubsp}
Let $Q\in\scrI_{1,\loc}(E,\mu)$ be the operator of orthogonal projection
onto a closed subspace $L\in L_2(E,\mu)$. Let $E'\subset E$ be a Borel subset such that
$\chi_{E \setminus E'}Q\chi_{E \setminus E'}\in\scrI_1(E,\mu)$ and that for any function $\varphi\in L$, the equality $\chi_{E'}\varphi=0$ implies $\varphi=0$. Then the subspace $\chi_{E'}L$ is closed
in $L_2(E, \mu)$.\end{proposition}

The subspace $H$ and the Borel subset $E_0$ therefore define an infinite determinantal measure $\mathbb{B}=\mathbb{B}(H,E_0)$. The measure $\mathbb{B}(H,E_0)$ is indeed infinite by
 Proposition \ref{suf-inf}.

\subsubsection{Multiplicative functionals of finite-rank perturbations}

Proposition \ref{mult-he1} now has the following immediate
\begin{corollary}\label{fin-rank-mult}
Let $L,V$ and $E_0$ induce an infinite determinantal measure $\mathbb{B}$.
Let $g\colon E\to(0,1]$ be a positive measurable function. If
\begin{enumerate}
\item $\sqrt{g}V\subset L_2(E,\mu)$\,;
\item $\sqrt{1-g}\Pi\sqrt{1-g}\in\scrI_1(E,\mu)$\,,
\end{enumerate}
then the multiplicative functional $\Psi_g$ is $\mathbb{B}$-almost surely
positive and integrable with respect to $\mathbb{B}$, and we have
$$\frac{\Psi_g\mathbb{B}}{\displaystyle \int\Psi_g\,d\mathbb{B}}=\Prob_{\Pi^g}\;,$$
where $\Pi^g$ is the operator of orthogonal projection onto the closed
subspace $\sqrt{g}L+\sqrt{g}V$.
\end{corollary}

\subsection{Example: the infinite Bessel point process}

We are now ready to prove Proposition \ref{inf-bessel} on the existence of the infinite Bessel point process $\widetilde{\mathbb{B}}^{(s)}$, $s\leq-1$.
We first need the following property of the usual Bessel point process $\widetilde{J}_s$, $s>-1$.
As before, let $\widetilde{L}_s$ be the range of the projection operator $\widetilde{J}_s$.

\begin{lemma} Let $s>-1$ be arbitrary. Then
\begin{enumerate}
\item For any $R>0$ the subspace $\chi_{(R,+\infty)}\widetilde{L}_s$ is closed in $L_2\bigl((0,+\infty),\Leb\bigr)$, and the corresponding projection operator $\widetilde{J}_{s,R}$ is locally of trace class;
\item For any $R>0$ we have
$$\Prob_{\widetilde{J}_s}\left(\Conf\left((0,+\infty),(R,+\infty)\right)\right)>0,$$
and
$$\frac{\left.\Prob_{\widetilde{J}_s}\right|_{\Conf\left((0,+\infty),(R,+\infty)\right)}} {\Prob_{\widetilde{J}_s}\left(\Conf\left((0,+\infty),(R,+\infty)\right)\right)}= \Prob_{\widetilde{J}_{s,R}}\,.$$
\end{enumerate}
\end{lemma}

\begin{proof} First, for any $R>0$ we clearly have
$$\int\limits_0^R\widetilde{J}_{s}(x,x)\,dx<+\infty$$
or, equivalently,
$$\chi_{(0,R)}\,\widetilde{J}_{s}\,\chi_{(0,R)}\in\scrI_1\bigl((0,+\infty),\Leb\bigr).$$

The Lemma follows now from the unique extension property of the Bessel point process.
The Lemma is proved completely.
\end{proof}

Now let $s\leq-1$ and recall that $n_s\in\mathbb{N}$ is defined by the relation $$\displaystyle\frac s2+n_s\in\left(-\frac12,\frac12\right].$$
Let
$$\check{V}^{(s)}=\Span \left(y^{s/2},y^{s/2+1},\dots,\frac{J_{s+2n_s-1}\left(\sqrt{y}\right)}{\sqrt{y}}\right).$$

\begin{proposition}\label{lin-ind} We have $\dim\check{V}^{(s)}=n_s$ and for any $R>0$ we have
$$\chi_{(0,R)}\widetilde{V}^{(s)}\cap L_2\bigl((0,+\infty),\Leb\bigr)=0.$$
\end{proposition}

Proof. The following argument has been suggested by Yanqi Qiu. By definition of the Bessel kernel, every function lying in $L^{s+2n_s}$
is in fact a restriction onto ${\mathbb R}_+$ of a harmonic function defined on the half-plane $\{z: \Re(z)>0\}$.
The desired claim follows now from the uniqueness theorem for harmonic functions.

Proposition \ref{lin-ind} immediately implies the existence of the infinite Bessel point process ${\tilde{\mathbb B}}^{(s)}$ and
concludes the proof of Proposition \ref{inf-bessel}.

Effectuating the change of variable $y=4/x$,
we also establish the existence of the modified infinite Bessel point process ${\mathbb B}^{(s)}$.

Furthermore, using the characterization of multiplicative functionals of infinite determinantal measures given by
Proposition \ref{mult-he1} and Corollary \ref{fin-rank-mult}, we arrive at the proof  of
Propositions \ref{expbetafin}, \ref{expbetal}, \ref{expbetacharac}.

\appendix

\section{The Jacobi Orthogonal Polynomial Ensemble}
\subsection{Jacobi polynomials}

Let $\alpha, \beta> -1$, and let $P_n^{(\alpha, \beta)}$ be the
standard Jacobi orthogonal polynomials, namely, polynomials on the unit interval $[-1,1]$  orthogonal with
weight
    $$(1-u)^{\alpha} (1+u)^{\beta}$$
and normalized by the condition
$$
P_n^{(\alpha, \beta)}(1)=\frac{\Gamma(n+\alpha+1)}{\Gamma(n+1)\Gamma(\alpha+1)}.
$$
Recall that the leading term $k_n^{(\alpha, \beta)}$ of $P_n^{(\alpha, \beta)}$ is given (see e.g. (4.21.6) in Szeg{\"o} \cite{Szego}) by the formula
\[
k_n^{(\alpha, \beta)} = \frac{\Gamma(2n+\alpha+\beta+1)}{2^n\cdot \Gamma(n+1) \cdot \Gamma(n+\alpha+\beta+1)}
\]
while for the square of the norm we have
\begin{multline*}
h_n^{(\alpha, \beta)} = \int\limits_{-1}^{1} \left( P_n^{(\alpha, \beta)}(u) \right)^2 \cdot (1-u)^{\alpha}(1+u)^{\beta} \, du =\\
= \frac{2^{\alpha+\beta+1}}{2n+\alpha+\beta+1} \frac{\Gamma(n+\alpha+1)\Gamma(n+\beta+1)}{\Gamma(n+1)\Gamma(n+\alpha+\beta+1)}.
\end{multline*}
Denote by $\tilde K_n^{(\alpha, \beta)} (u_1, u_2)$ the $n$-th Christoffel-Darboux kernel of the Jacobi orthogonal  polynomial ensemble:
\begin{equation*}
\tilde K_n^{(\alpha, \beta)} (u_1,u_2) = \sum_{l=0}^{n-1} \frac{P_l^{(\alpha, \beta)}(u_1)\cdot P_l^{(\alpha, \beta)} (u_2)}{h_l^{(\alpha, \beta)}}(1-u_1)^{\alpha/2} (1+u_1)^{\beta/2}(1-u_2)^{\alpha/2} (1+u_2)^{\beta/2}.
\end{equation*}
The Christoffel-Darboux formula gives an equivalent representation for the kernel $\tilde K_n^{(\alpha, \beta)}$:
\begin{multline}\label{CDJac-gen}
\tilde K^{(\alpha, \beta)}_n (u_1, u_2) =\\= \frac{2^{-\alpha-\beta}}{2n+\alpha+\beta}\frac{\Gamma(n+1)\Gamma(n+\alpha+\beta+1)}{\Gamma(n+\alpha)\Gamma(n+\beta)}\cdot(1-u_1)^{\alpha/2} (1+u_1)^{\beta/2}(1-u_2)^{\alpha/2} (1+u_2)^{\beta/2}\times\\
\times \frac{P_{n}^{(\alpha, \beta)}(u_1)P_{n-1}^{(\alpha, \beta)}(u_2)-P_n^{(\alpha, \beta)}(u_2)P_{n-1}^{(\alpha, \beta)}(u_1)}{u_1-u_2}.
\end{multline}

\subsection{The recurrence relation between Jacobi polynomials}
We have the following recurrence relation between the Christoffel-Darboux kernels $\tilde K_{n+1}^{(\alpha, \beta)}$
and $\tilde K_{n}^{(\alpha+2, \beta)}$.

\begin{proposition} \label{rec-form-jac} For any $\alpha, \beta>-1$ we have
\begin{multline}\label{recreljacpol}
\tilde K_{n+1}^{(\alpha, \beta)}(u_1, u_2)=\\=
\frac{\alpha+1}{2^{\alpha+\beta+1}}\frac{\Gamma(n+1)\Gamma(n+\alpha+\beta+2)}
{\Gamma(n+\alpha+2)\Gamma(n+\beta+1)}P_n^{(\alpha+1, \beta)}(u_1)(1-u_1)^{\alpha/2} (1+u_1)^{\beta/2}\times\\ \times P_n^{(\alpha+1, \beta)}(u_2)(1-u_2)^{\alpha/2} (1+u_2)^{\beta/2}+\\+ \tilde K_{n}^{(\alpha+2, \beta)}(u_1, u_2).
\end{multline}
\end{proposition}

{\bf Remark.} The recurrence relation \eqref{recreljacpol} can of course be
taken to the scaling limit to yield a similar recurrence relation for Bessel kernels:
the Bessel kernel with parameter $s$ is thus a rank one perturbation of the Bessel
kernel with parameter $s+2$.
This is also easily esablished directly:
 using the recurrence relation
\begin{equation*}
J_{s+1}(x)=\frac{2s}{x}J_s(x)-J_{s-1}(x)
\end{equation*}
for Bessel functions,
one immediately obtains the desired recurrence  relation
\begin{equation*}
{\tilde J}_s(x,y)={\tilde J}_{s+2}(x,y)+\frac{s+1}{\sqrt{xy}}J_{s+1}(\sqrt{x})J_{s+1}(\sqrt{y})
\end{equation*}
for the Bessel kernels.

Proof of Proposition \ref{rec-form-jac}. The routine calculation is included for completeness. We use standard recurrence relations for Jacobi polynomials.
First, we use the relation (see e.g. (4.5.4) in Szeg{\"o} \cite{Szego})
$$
(n+\frac{\alpha+\beta}{2}+1)(u-1)P_n^{(\alpha+1, \beta)}(u)= (n+1) P_{n+1}^{(\alpha, \beta)}(u)-(n+\alpha+1)P_n^{(\alpha, \beta)}(u)
$$
to arrive at the equality
\begin{multline}\label{receq1}
\frac{P_{n+1}^{(\alpha, \beta)}(u_1)P_{n}^{(\alpha, \beta)}(u_2)-P_{n+1}^{(\alpha, \beta)}(u_2)P_{n}^{(\alpha, \beta)}(u_1)}{u_1-u_2}=\\
=\frac{2n+\alpha+\beta+2}{2(n+1)}\frac{(u_1-1)P_{n}^{(\alpha+1, \beta)}(u_1)P_{n}^{(\alpha, \beta)}(u_2)-(u_2-1)P_n^{(\alpha+1, \beta)}(u_2)P_{n}^{(\alpha, \beta)}(u_1)}{u_1-u_2}.
\end{multline}

We next apply the relation (see e.g. (22.7.18) in Abramowitz--Stegun \cite{AS})
$$
(2n+\alpha+\beta+1)P_n^{(\alpha, \beta)}(u)=(n+\alpha+\beta+1)P_n^{(\alpha+1, \beta)}(u)-
(n+\beta)P_{n-1}^{(\alpha+1, \beta)}(u)
$$
to arrive at the equality
\begin{multline}\label{receq2}
\frac{(u_1-1)P_{n}^{(\alpha+1, \beta)}(u_1)P_{n}^{(\alpha, \beta)}(u_2)-(u_2-1)P_n^{(\alpha+1, \beta)}(u_2)P_{n}^{(\alpha, \beta)}(u_1)}{u_1-u_2}=\\
=\frac{n+\alpha+\beta+1}{2n+\alpha+\beta+1}P_n^{(\alpha+1, \beta)}(u_1)P_n^{(\alpha+1, \beta)}(u_2)+\\+
\frac{n+\beta}{2n+\alpha+\beta+1}\frac{(1-u_1)P_{n}^{(\alpha+1, \beta)}(u_1)P_{n-1}^{(\alpha+1, \beta)}(u_2)-(1-u_2)P_n^{(\alpha+1, \beta)}(u_2)P_{n-1}^{(\alpha+1, \beta)}(u_1)}{u_1-u_2}.
\end{multline}
Using next the recurrence relation
$$
(n+\frac{\alpha+\beta+1}{2})(1-u)P_{n-1}^{(\alpha+2, \beta)}(u)=(n+\alpha+1)P_{n-1}^{(\alpha+1, \beta)}(u)-nP_n^{(\alpha+1, \beta)}(u),
$$
we arrive at the equality
\begin{multline}\label{receq3}
\frac{(1-u_1)P_{n}^{(\alpha+1, \beta)}(u_1)P_{n-1}^{(\alpha+1, \beta)}(u_2)-(1-u_2)P_n^{(\alpha+1, \beta)}(u_2)P_{n-1}^{(\alpha+1, \beta)}(u_1)}{u_1-u_2}=\\
=-\frac{n}{n+\alpha+1}P_n^{(\alpha+1, \beta)}(u_1)P_n^{(\alpha+1, \beta)}(u_2)+\\+\frac{2n+\alpha+\beta+1}{2(n+\alpha+1)}
(1-u_1)(1-u_2)\frac{P_{n}^{(\alpha+1, \beta)}(u_1)P_{n-1}^{(\alpha+2, \beta)}(u_2)-P_n^{(\alpha+1, \beta)}(u_2)P_{n-1}^{(\alpha+2, \beta)}(u_1)}{u_1-u_2}.
\end{multline}
Combining \eqref{receq2} and \eqref{receq3}, we obtain
\begin{multline}\label{receq4}
\frac{(u_1-1)P_{n}^{(\alpha+1, \beta)}(u_1)P_{n}^{(\alpha, \beta)}(u_2)-(u_2-1)P_n^{(\alpha+1, \beta)}(u_2)P_{n}^{(\alpha, \beta)}(u_1)}{u_1-u_2}=\\
=\frac{(\alpha+1)(2n+\alpha+\beta+1)}{(n+\alpha+1)(2n+\alpha+\beta+1)}P_n^{(\alpha+1, \beta)}(u_1)P_n^{(\alpha+1, \beta)}(u_2)+\\
+\frac{n+\beta}{2(n+\alpha+1)} (1-u_1)(1-u_2)\frac{P_{n}^{(\alpha+1, \beta)}(u_1)P_{n-1}^{(\alpha+2, \beta)}(u_2)-P_n^{(\alpha+1, \beta)}(u_2)P_{n-1}^{(\alpha+2, \beta)}(u_1)}{u_1-u_2}.
\end{multline}
Using the recurrence relation
$$
(2n+\alpha+\beta+2)P_n^{(\alpha+1, \beta)}(u)=(n+\alpha+\beta+2)P_n^{(\alpha+2, \beta)}(u)-(n+\beta)P_{n-1}^{(\alpha+2, \beta)}(u),
$$
we now arrive at the relation
\begin{multline}\label{receq5}
\frac{P_{n}^{(\alpha+1, \beta)}(u_1)P_{n-1}^{(\alpha+2, \beta)}(u_2)-P_n^{(\alpha+1, \beta)}(u_2)P_{n-1}^{(\alpha+2, \beta)}(u_1)}{u_1-u_2}=\\
=\frac{n+\alpha+\beta+2}{2n+\alpha+\beta+2} \frac{P_{n}^{(\alpha+2, \beta)}(u_1)P_{n-1}^{(\alpha+2, \beta)}(u_2)-P_n^{(\alpha+2, \beta)}(u_2)P_{n-1}^{(\alpha+2, \beta)}(u_1)}{u_1-u_2}.
\end{multline}
Combining \eqref{receq1}, \eqref{receq4}, \eqref{receq5} and recalling the definition \eqref{CDJac-gen} of Christoffel---Darboux kernels, we conclude the proof of Proposition \ref{rec-form-jac}.

As above, given a finite family of functions $f_1, \dots, f_N$ on the unit interval or on the real line, we let $\Span(f_1, \dots, f_N)$ stand for the vector space these functions span.
For $\alpha,\beta\in {\mathbb R}$ introduce the subspace
\begin{multline*}
L_{Jac}^{(\alpha, \beta,n)}=\Span((1-u)^{\alpha/2}(1+u)^{\beta/2}, (1-u)^{\alpha/2}(1+u)^{\beta/2}u, \dots \\
 \dots, (1-u)^{\alpha/2}(1+u)^{\beta/2}u^{n-1}).
\end{multline*}

For $\alpha, \beta>-1$,
Proposition \ref{rec-form-jac} yields the following orthogonal direct-sum decomposition
\begin{equation}\label{rec-dec-ab}
L_{Jac}^{(\alpha, \beta,n)}={\mathbb C}P_n^{(\alpha+1, \beta)}\oplus L_{Jac}^{(\alpha+2, \beta,n-1)}.
\end{equation}
Though the corresponding spaces are no longer subspaces in $L_2$, the relation \eqref{rec-dec-ab} is still valid for all  $\alpha\in (-2,-1]$; in reformulating it, it is, however,  more convenient for us to shift $\alpha$ by $2$.
\begin{proposition}\label{ab-span}
For all $\alpha>0$, $\beta>-1$, $n\in {\mathbb N}$ we have
$$
L_{Jac}^{(\alpha-2, \beta,n)}={\mathbb C}P_n^{(\alpha-1, \beta)}\oplus L_{Jac}^{(\alpha, \beta,n-1)}.
$$
\end{proposition}
Proof.
Let $Q_n^{(\alpha, \beta)}$ be the  function of the
second kind corresponding to the Jacobi polynomial $P_n^{(\alpha, \beta)}$. By Szeg{\"o}, \cite{Szego}, formula (4.62.19), for any $u\in(-1,1)$, $v>1$ we have

\begin{multline*}
\sum\limits_{l=0}^n \frac{(2l+\alpha+\beta+1)}{2^{\alpha+\beta+1}}\frac{\Gamma(l+1)\Gamma(l+\alpha+\beta+1)}{\Gamma(l+\alpha+1)\Gamma(l+\beta+1)}P_l^{(\alpha)}(u)Q_l^{(\alpha)}(v)=\\
=\frac12\frac{(v-1)^{-\alpha}(v+1)^{-\beta}}{(v-u)}+\\
+\frac{2^{-\alpha-\beta}}{2n+\alpha+\beta+2}\frac{\Gamma(n+2)\Gamma(n+\alpha+\beta+2)}{\Gamma(n+\alpha+1)\Gamma(n+\beta+1)}
\frac{P_{n+1}^{(\alpha, \beta)}(u)Q_n^{(\alpha, \beta)}(v)-Q_{n+1}^{(\alpha, \beta)}(v)P_n^{(\alpha, \beta)}(u)}{v-u}\,.
\end{multline*}
Take the limit $v\to1$, and recall from Szeg{\"o} \cite{Szego}, formula (4.62.5),
the following asymptotic expansion as $v\to 1$
for the Jacobi function of the second kind
$$
Q_n^{(\alpha)}(v) \sim \frac{2^{\alpha-1}\Gamma(\alpha)\Gamma(n+\beta+1)}{\Gamma(n+\alpha+\beta+1)} (v-1)^{-\alpha}.
$$
Recalling the recurrence formula (22.7.19)in \cite{AS}:
$$
P_{n+1}^{(\alpha-1,\beta)}(u)=(n+\alpha+\beta+1)P_{n+1}^{(\alpha, \beta)}-(n+\beta+1)P_n^{(\alpha, \beta)}(u)
$$
we arrive at
the relation
$$
\frac1{1-u}+\frac{\Gamma(\alpha)\Gamma(n+2)}{\Gamma(n+\alpha+1)}P_{n+1}^{(\alpha-1, \beta)}\in L_{Jac}^{(\alpha, \beta, n)},
$$
which immediately implies Proposition \ref{ab-span}.

 Now take $s>-1$ and, for brevity, write  $P_n^{(s)}=P_n^{(\alpha, \beta)}$.
The leading term $k_n^{(s)}$ of $P_n^{(s)}$ is given by the formula
\[
k_n^{(s)} = \frac{\Gamma(2n+s+1)}{2^n\cdot n! \cdot \Gamma(n+s+1)}
\]
while for the square of the norm we have
\[
h_n^{(s)} = \int\limits_{-1}^{1} \left( P_n^{(s)}(u) \right)^2 \cdot (1-u)^s \, du = \frac{2^{s+1}}{2n+s+1}.
\]
Denote by $\tilde K_n^{(s)} (u_1, u_2)$ the corresponding $n$-th Christoffel-Darboux kernel :
\begin{equation}\label{CDJ}
\tilde K_n^{(s)} (u_1,u_2) = \sum_{l=1}^{n-1} \frac{P_l^{(s)}(u_1)\cdot P_l^{(s)} (u_2)}{h_l^{(s)}}(1-u_1)^{s/2}(1-u_2)^{s/2}.
\end{equation}
The Christoffel-Darboux formula gives an equivalent representation for the kernel $\tilde K_n^{(s)}$:
\begin{equation}\label{CDJ2}
\tilde K^{(s)}_n (u_1, u_2) = \frac{n(n+s)}{2^s(2n+s)}\cdot(1-u_1)^{s/2} \cdot (1-u_2)^{s/2}\cdot\frac{P_n^{(s)}(u_1)P_{n-1}^{(s)}(u_2)-P_n^{(s)}(u_2)P_{n-1}^{(s)}(u_1)}{u_1-u_2}
\end{equation}
\subsection{The Bessel kernel}
Consider the half-line $(0, +\infty)$ endowed with the standard Lebesgue measure
$\Leb$.
Take $s>-1$ and consider the standard Bessel kernel
\begin{equation*}
{\tilde J}_s(y_1,y_2)=\frac{\sqrt{y_1}J_{s+1}(\sqrt{y_1})J_s(\sqrt{y_2})-\sqrt{y_2}J_{s+1}(\sqrt{y_2})J_s(\sqrt{y_1})}{2(y_1-y_2)}
\end{equation*}
(see, e.g., page 295 in Tracy and Widom \cite{TracyWidom}).

An alternative integral representation for the kernel ${\tilde J}_s$ has the form
\begin{equation}\label{bessel-int}
{\tilde J}_s(y_1,y_2)=\frac14\int\limits_0^1 J_s(\sqrt{ty_1})J_s(\sqrt{ty_2})dt
\end{equation}
(see, e.g., formula (2.2) on  page 295 in Tracy and Widom \cite{TracyWidom}).

As \eqref{bessel-int} shows, the kernel ${\tilde J}_s$ induces on $L_2((0, +\infty), {\Leb})$ the operator of orthogonal projection onto the subspace of functions whose
 Hankel transform is supported in $[0,1]$ (see \cite{TracyWidom}).

\begin{proposition}\label{tilde-kernel-kns-js}
For any $s>-1$, as $n\to\infty$, the kernel ${\tilde K}_n^{(s)}$
converges
to the kernel ${\tilde J}_{s}$ uniformly in the totality of variables
on compact subsets of
$(0, +\infty)\times (0, +\infty)$.
\end{proposition}

Proof. This is an immediate corollary of the classical Heine-Mehler asymptotics for Jacobi orthogonal polynomials, see e.g.
Chapter 8 in Szeg{\"o} \cite{Szego}. Note that the uniform
convergence in fact takes place on arbitrary simply connected compact subsets of
$({\mathbb C}\setminus 0)\times {\mathbb C}\setminus 0$.

\section{Spaces of configurations and determinantal point processes}

\subsection{Spaces of configurations}\label{subsec-space-conf}

Let $E$ be a locally compact complete metric space.

A configuration $X$ on $E$ is a collection of points, called {\it particles} considered without regard to order; the main assumption is that particles do not accumulate anywhere in $E$, or, equivalently that a bounded subset of $E$ contain only finitely many particles of a configuration.

To a configuration $X$ assign the Radon measure
$$
\sum\limits_{x \in X} \delta_x
$$
where the summation takes place over all particles of $X$.
Conversely, any purely atomic Radon measure on $E$ is given by a configuration. The space $\Conf(E)$ of configurations on $E$ is thus identified with a closed subset of integer-valued
Radon measures on $E$ in the space of all Radon measures on $E$. This identification endows $\Conf(E)$ with the structure of a complete separable metric space, which, however, is not locally compact.

The Borel structure on $\Conf(E)$ can be equivalently defined as follows. For a bounded Boral subset $B \subset E$, introduce a function
$$
\#_{B}\colon \Conf(E) \to \mathbb{R}
$$
that assigns to a configuration $X$ the number of its particles that lie in $B$. The family of functions $\#_B$ over alll bounded Borel subsets $B \subset E$ determines the Borel structure on $\Conf(E)$; in particular, to define a probability measure on $\Conf(E)$ it is necessary and sufficient to define the joint distributions of the random variables $\#_{B_1}, \dots, \#_{B_k}$ over all finite collections of disjoint bounded Borel subsets $B_1, \dots, B_k \subset E$.

\subsection{ Weak topology on the space of probability measures on the space of configurations}

The space $\Conf(E)$ is  endowed with a natural structure of a complete
separable metric space, and the space $\Mfin(\Conf(E))$  of finite Borel measures on the space of configurations is consequently also
a complete separable metric space with respect to the weak topology.

Let $\varphi:E\to\mathbb{R}$ be a compactly supported continuous  function.
Define a measurable function
$\#_{\varphi}:\Conf(E)\to\mathbb{R}$ by the formula
$$\#_{\varphi}(X)=\sum\limits_{x\in X}\varphi(x).$$
For a bounded Borel subset $B\subset E$, of course, we have
$\#_B=\#_{\chi_B}$.

Since the Borel sigma-algebra on $\Conf(E)$ coincides with the sigma-algebra generated by the integer-valued
random variables $\#_B$ over all bounded Borel subsets $B\subset E$,  it also coincides with the sigma-algebra generated by the random variables $\#_{\varphi}$ over all compactly supported continuous  functions $\varphi:E\to\mathbb{R}$.
Consequently, we have the following
\begin{proposition}\label{fin-distr-unique}
 A Borel probability measure $\Prob\in \Mfin(\Conf(E))$ is uniquely determined
by the joint distributions of all finite collections $$\#_{\varphi_1}, \#_{\varphi_2}, \dots, \#_{\varphi_l}$$
over all continuous  functions $\varphi_1, \dots, \varphi_l:E\to\mathbb{R}$ with disjoint compact supports.
\end{proposition}
The weak topology on $\Mfin(\Conf(E))$ admits the following  characterization in terms of the said finite-dimensional distributions (see Theorem 11.1.VII in vol.2 of \cite{DVJ}). Let $\Prob_n, n\in{\mathbb N}$ and $\Prob$ be Borel probability measures on $\Conf(E)$. Then the measures $\Prob_n$ converge to $\Prob$ weakly as $n\to\infty$ if and only if for any finite collection $\varphi_1, \dots, \varphi_l$ of  continuous  functions  with disjoint compact supports
the joint distributions of the random variables
$\#_{\varphi_1}, \dots, \#_{\varphi_l}$ with respect to $\Prob_n$ converge, as $n\to\infty$,
to the joint distribution of  $\#_{\varphi_1}, \dots, \#_{\varphi_l}$ with respect to $\Prob$; convergence of joint distributions being understood according to the weak topology
on the space of Borel probability measures on ${\mathbb R}^l$.

\subsection{Spaces of locally  trace class operators}\label{subsec-loc-tr-cls}

Let $\mu$ be a sigma-finite Borel measure on $E$.

Let $\scrI_{1}(E,\mu)$ be the ideal of trace class operators
${\widetilde K}\colon L_2(E,\mu)\to L_2(E,\mu)$ (see volume~1 of~\cite{ReedSimon} for
the precise definition); the symbol
$||{\widetilde K}||_{\scrI_1}$ will stand for the
$\scrI_{1}$-norm of the operator ${\widetilde K}$.
Let $\scrI_{2}(E,\mu)$ be the
ideal of Hilbert-Schmidt operators ${\widetilde K}\colon
L_2(E,\mu)\to L_2(E,\mu)$; the symbol $||{\widetilde
K}||_{\scrI_2}$ will stand for the $\scrI_{2}$-norm of
the operator ${\widetilde K}$.

Let  $\scrI_{1,\loc}(E,\mu)$ be the space of operators  $K\colon L_2(E,\mu)\to L_2(E,\mu)$
such that for any bounded Borel subset $B\subset E$
we have $$\chi_BK\chi_B\in\scrI_1(E,\mu).$$
Again, we endow the space $\scrI_{1,\loc}(E,\mu)$
with a countable family of semi-norms
\begin{equation*}
||\chi_BK\chi_B||_{\scrI_1}
\end{equation*}
where, as before, $B$ runs through an exhausting family $B_n$ of bounded sets.

\subsection{Determinantal Point Processes}

A Borel probability measure $\Prob$ on
$\Conf(E)$ is called
\textit{determinantal} if there exists an operator
$K\in\scrI_{1,\loc}(E,\mu)$ such that for any bounded measurable
function $g$, for which $g-1$ is supported in a bounded set $B$,
we have
\begin{equation}
\label{eq1}
\mathbb{E}_{\Prob}\Psi_g
=\det\biggl(1+(g-1)K\chi_{B}\biggr).
\end{equation}
The Fredholm determinant in~\eqref{eq1} is well-defined since
$K\in \scrI_{1, \loc}(E,\mu)$.
The equation (\ref{eq1}) determines the measure $\Prob$ uniquely.
For any
pairwise disjoint bounded Borel sets $B_1,\dotsc,B_l\subset E$
and any  $z_1,\dotsc,z_l\in {\mathbb C}$ from (\ref{eq1}) we  have
$\mathbb{E}_{\Prob}z_1^{\#_{B_1}}\cdots z_l^{\#_{B_l}}
=\det\biggl(1+\sum\limits_{j=1}^l(z_j-1)\chi_{B_j}K\chi_{\sqcup_i B_i}\biggr).$

For further results and background on determinantal point processes, see e.g. \cite{BorOx},  \cite{HoughEtAl}, \cite{Lyons},
\cite{LyonsSteif}, \cite{Lytvynov}, \cite{ShirTaka0},  \cite{ShirTaka1}, \cite{ShirTaka2}, \cite{Soshnikov}.

If $K$ belongs to
$\scrI_{1,\loc}(E,\mu)$, then, throughout the paper, we denote
the corresponding determinantal measure by
$\Prob_K$. Note that $\Prob_K$ is uniquely defined by~$K$,
but different operators may yield the same measure.
By the Macch{\` \i}---Soshnikov theorem~\cite{Macchi}, \cite{Soshnikov}, any
Hermitian positive contraction that belongs
to the class~$\scrI_{1,\loc}(E,\mu)$ defines a determinantal point process.

\subsection{Change of variables}\label{subsec-ch-of-var}

Let $F:E\to E$ be a homeomorphism.  The  homeomorphism $F$ induces a homeomorphism of the space $\Conf(E)$, for which, slightly abusing notation, we keep the same symbol: given $X\in\Conf(E)$, the particles of the configuration
$F(X)$ have the form $F(x)$ over all $x\in X$.
Now, as before, let $\mu$ be a sigma-finite measure on $E$, and  let $\Prob_K$ be the determinantal measure induced by an operator $K\in \scrI_{1, \loc}(E, \mu)$.
Let the operator  $F_*K$ be defined by the formula $F_*K(f)=K(f\circ F)$.

Assume now that the measures $F_*\mu$ and $\mu$ are equivalent,
and consider the operator
$$
K^F=\sqrt{\frac{dF_*\mu}{d\mu}}F_*K\sqrt{\frac{dF_*\mu}{d\mu}}.
$$
Note that if $K$ is self-adjoint, then so is $K^F$.
If $K$ is given by the kernel $K(x,y)$, then $K^F$ is given by the kernel
$$
K^F(x,y)=\sqrt{\frac{dF_*\mu}{d\mu}\left(x\right)}K(F^{-1}x, F^{-1}y)
\sqrt{\frac{dF_*\mu}{d\mu}\left( y\right)}.
$$
Directly from the definitions we now have the following
\begin{proposition}
The action of the homeomorphism $F$ on the determinantal measure $\Prob_K$
is given by the formula
$$
F_*\Prob_K=\Prob_{K^F}.
$$
\end{proposition}

Note that if $K$ is the operator of orthogonal projection onto the closed subspace
$L\subset L_2(E, \mu)$, then, by definition, the operator
 $K^F$ is the operator of orthogonal projection onto the closed subspace
$$\{\varphi\circ F^{-1}(x)\sqrt{\frac{dF_*\mu}{d\mu}\left(x\right)}\}\subset L_2(E, \mu).$$

\subsection{Multiplicative functionals on spaces of configurations}

Let $g$ be a non-negative measurable function on $E$, and introduce the
{\it multiplicative functional} $\Psi_g:\Conf(E)\to\mathbb{R}$ by the formula
 \begin{equation*}
 \Psi_g(X)=\prod\limits_{x\in X}g(x).
 \end{equation*}
If the infinite product
$\prod\limits_{x\in X}g(x)$ absolutely converges to $0$ or to $\infty$, then we set, respectively,
$\Psi_g(X)=0$ or $\Psi_g(X)=\infty$. If the product in the right-hand side fails to converge absolutely,
then the multiplicative functional is not defined.

\subsection{Multiplicative functionals of determinantal point processes}

At the centre of the construction of infinite determinantal measures lie the results of \cite{Buf-umn}, \cite{Buf-CIRM}  that can informally be summarized as follows:
a determinantal measure times a multiplicative functional is again a determinantal measure.
In other words, if $\Prob_K$ is a determinantal measure on $\Conf(E)$ induced by the
operator $K$ on $L_2(E,\mu)$, then, under certain additional assumptions,
it is shown in \cite{Buf-umn}, \cite{Buf-CIRM}  that the
measure ${ \Psi_g\Prob_K}$
after normalization yields a determinantal point process.

As before, let $g$ be a non-negative measurable function on $E$. If the operator $1+(g-1)K$ is invertible, then we set
$${\mathfrak B}(g, K)=g K(1+{(g-1)}K)^{-1},\qquad
{\widetilde {\mathfrak B}}(g, K)= {\sqrt{g}}K(1+{(g-1)}K)^{-1}{\sqrt{g}}.$$
By definition, ${\mathfrak B}(g,K),{\widetilde {\mathfrak B}}(g,K)\in \scrI_{1,\loc}(E,\mu)$ since $K\in \scrI_{1,\loc}(E,\mu)$, and, if $K$ is self-adjoint, then
so is ${\widetilde {\mathfrak B}}(g,K)$.

We now recall a few propositions from \cite{Buf-CIRM}.
\begin{proposition}
\label{pr1-bis}
Let $K\in \scrI_{1,\loc}(E,\mu)$ be a self-adjoint positive contraction, and let $\Prob_K$ be
the corresponding determinantal measure on $\Conf(E)$. Let
$g$ be a nonnegative bounded measurable
function on~$E$ such that
\begin{equation}
\label{gkint}
\sqrt{g-1}K\sqrt{g-1}\in\scrI_{1}(E,\mu)
\end{equation}
and that the operator
$1+{(g-1)}K$ is invertible. Then
\begin{enumerate}
\item we have $\Psi_g\in L_1(\Conf(E),\Prob_K)$ and
\begin{equation*}
\int\Psi_g\,d\Prob_K=\det\Bigl(1+\sqrt{g-1}K\sqrt{g-1}\Bigr)>0;
\end{equation*}
\item the operators
${\mathfrak B}(g, K), {\widetilde {\mathfrak B}}(g, K)$ induce on $\Conf(E)$ a determinantal measure
$\Prob_{{\mathfrak B}(g, K)}=\Prob_{{\widetilde {\mathfrak B}}(g, K)}$ satisfying
\begin{equation*}
\Prob_{{\mathfrak B}(g, K)}=\frac{\displaystyle \Psi_g\Prob_K}
{\displaystyle \int\limits_{\Conf(E)}\Psi_g\,d\Prob_K}.
\end{equation*}
\end{enumerate}
\end{proposition}

{\bf Remark.} Since (\ref{gkint}) holds and $K$ is self-adjoint,  the operator $1+{(g-1)}K$ is
invertible if and only if the operator  $1+\sqrt{g-1}K\sqrt{g-1}$ is invertible.

If $Q$ is a projection operator, then the operator $\tilde {\mathfrak B}(g,Q)$ admits the following description.

\begin{proposition}
Let $L\subset L_2(E,\mu)$ be a closed subspace, and let
$Q$ be the operator of orthogonal projection onto $L$.
Let $g$ be a bounded measurable function such that the operator $1+(g-1)Q$ is invertible. Then the operator $\tilde {\mathfrak B}(g,Q)$ is the operator of orthogonal projection onto the closure of the subspace $\sqrt{g}L$.
\end{proposition}

We now consider the particular case when  $g$ is a characteristic function of a Borel subset.
In much the same way as before, if
$E'\subset E$ is a Borel subset such that  the subspace $\chi_{E'}L$ is closed (recall that a sufficient condition for that is provided in Proposition \ref{closedsubsp}),  then we
set $Q^{E'}$ to be
the operator of orthogonal projection onto the closed subspace $\chi_{E'}L$.

Proposition \ref{pr1-bis} now yields the following
\begin{corollary}\label{indsubset}
Let $Q\in\scrI_{1,\loc}(E,\mu)$ be the operator of orthogonal projection
onto a closed subspace $L\in L_2(E,\mu)$. Let $E'\subset E$ be a Borel subset such that
$\chi_{E\setminus E'}Q\chi_{E\setminus E'}\in\scrI_1(E,\mu)$.
Then
\begin{equation*}
\Prob_{Q}(\Conf(E, E'))=\det(1-\chi_{E\setminus E'}Q\chi_{E\setminus E'}).
\end{equation*}
Assume, additionally, that for any function $\varphi\in L$, the equality $\chi_{E'}\varphi=0$ implies $\varphi=0$.
Then  the subspace $\chi_{E'}L$ is closed, and we have
$$\Prob_{Q}(\Conf(E, E'))>0,  \  Q^{E'}\in\scrI_{1,\loc}(E,\mu),$$
and
\begin{equation*}
\frac{\Prob_{Q}|_{\Conf(E,E')}}{\Prob_{Q}(\Conf(E, E'))}=\Prob_{Q^{E'}}.
\end{equation*}
\end{corollary}

The induced measure of a determinantal measure onto the subset of configurations all whose particles lie in $E'$ is thus again a determinantal measure. In the case of a discrete phase space, related induced processes were considered by Lyons \cite{Lyons} and by Borodin and Rains \cite{BorRains}.

We now give a sufficient condition for the almost sure {\it
positivity} of a multiplicative
functional.
\begin{proposition} If $$\mu\left(\left\{x\in E\colon
g(x)=0\right\}\right)=0$$ and
$$\sqrt{|g-1|}K\sqrt{|g-1|}\in\scrI_1(E,\mu),$$ then
$$0<\Psi_g(X)<+\infty$$
for $\Prob_K$-almost all $X\in \Conf(E)$.
\end{proposition}
{Proof}. Our assumptions imply that for $\Prob_K$-almost
all $X\in \Conf(E)$ we have
$$\sum\limits_{x\in X}|g(x)-1|<+\infty\,,$$
which, in turn, is sufficient for absolute convergence of the infinite
product $\prod\limits_{x\in X}g(x)$
to a finite non-zero limit.

We also formulate a version of Proposition \ref{pr1-bis} in the special
case when the function $g$ does not assume values less than $1$. In this case the multiplicative
functional $\Psi_g$ is automatically non-zero, and we have
\begin{proposition} Let $\Pi\in\scrI_{1,\loc}(E,\mu)$ be the
operator of orthogonal projection
onto a closed subspace $H$, let $g$ be a bounded Borel function on $E$ satisfying $g(x)\geq 1$ for all $x\in E$, and
assume
$$\sqrt{g-1}\Pi\sqrt{g-1}\in\scrI_1(E,\mu)\,.$$
Then:
\begin{enumerate}
\item $\Psi_g\in L_1(\Conf(E),\Prob_{\Pi})$, and
$$\displaystyle \int\Psi_g\,d\Prob_{\Pi}=\det\left(1+\sqrt{g-1}\Pi\sqrt{g-1}\right);$$
\item we have
$$\frac{\Psi_g\Prob_{\Pi}}{\displaystyle \int\Psi_g\,d\Prob_{\Pi}}=\Prob_{\Pi^g},$$
where
$\Pi^g$ is the operator of orthogonal projection onto the subspace
$\sqrt{g}H$.
\end{enumerate}
\end{proposition}

\section{Construction of Pickrell Measures  and proof of Proposition \ref{rel-consis}}

First we recall that the Pickrell measures are naturally defined
on the space of {\it rectangular} $m\times n$-matrices.

Let $\Mat(m\times n, \mathbb C)$ be the space of $m \times n$ matrices with complex entries:
\[
\Mat (m \times n, \mathbb C) = \{ z = (z_{ij}),\; i=1,\dots,\, m;j=1,\dots, n \}
\]
Denote $dz$ the Lebesgue measure on $\Mat (m \times n, \mathbb C)$.

Take $s \in \mathbb R$. Let $m_0,n_0$ be such that $m_0+s>0, n_0+s>0$.
Following Pickrell, take $m>m_0$, $n>n_0$ and
introduce a measure $\mu_{m,n}^{(s)}$ on $\Mat (m \times n, \mathbb C)$ by the formula
\begin{equation*}
\mu_{m,n}^{(s)} = \const_{m,n}^{(s)}\cdot \det(1+z^*z)^{-m-n-s}\,dz,
\end{equation*}
where
\begin{equation*}
\const_{m,n}^{(s)} = \pi^{-mn} \cdot \prod_{l=m_0}^m \frac{\Gamma(l+s)}{\Gamma(n+l+s)}.
\end{equation*}
For $m_1\le m$, $n_1\le n$, let
\[
\pi^{m,n}_{m_1,n_1}: \Mat (m\times n, \mathbb C) \to \Mat(m_1\times n_1, \mathbb C)
\]
be the natural projection map.

\begin{proposition}\label{mn-consis}

Let $m,n\in\mathbb{N}$ be such that $s>-m-1$.
Then for any ${\tilde z}\in \Mat(n, {\mathbb C})$ we have
\begin{multline*}
\int\limits_{(\pi^{m+1,n}_{m,n})^{-1}({\tilde z})}
\det(1+{z}^*{z})^{-m-n-1-s}dz=\\
\pi^n\,\frac{\Gamma(m+1+s)}{\Gamma(n+m+1+s)}
\det(1+{{\tilde z}}^*{{\tilde z}})^{-m-n-s}.
\end{multline*}
\end{proposition}

Proposition \ref{rel-consis} is an immediate corollary of Proposition \ref{mn-consis}.

Proof of Proposition \ref{mn-consis}. As we noted in the Introduction, the following computation goes back to the classical work of Hua Loo-Keng \cite{Hua}.
Take $z\in\Mat\bigl((m+1)\times n,\mathbb{C}\bigr)$. Multiplying, if necessary, by a unitary matrix on the left and on the right, represent the matrix $\pi^{m+1,n}_{\ m,\,n}\,z=\widetilde{z}$ in diagonal form, with positive real entries on the diagonal: $\widetilde{z}_{ii}=u_i>0$, $i=1,\dots,n$, $\widetilde{z}_{ij}=0$~for~$i\neq j$.

Here we set $u_i=0$ for $i>\min(n,m)$. Denote $\xi_i=z_{m+1,i}$, $i=1,\dots,n$.
Write
$$\det\left(1+z^*z\right)^{-m-1-n-s}=\prod\limits_{i=1}^m(1+u_i^2)^{-m-1-n-s}\!\times \left(1+\xi^*\xi-\sum\limits_{i=1}^n\frac{|\xi_i|^2\,u_i^2}{1+u_i^2}\right)^{-m-1-n-s}\,.$$
We have
$$1+\xi^*\,\xi-\sum\limits_{i=1}^n\frac{|\xi_i|^2\,u_i^2}{1+u_i^2}= 1+\sum\limits_{i=1}^n\frac{|\xi_i|^2}{1+u_i^2}\,.$$

Integrating in $\xi$, we find
$$\int\left(1+\sum\limits_{i=1}^n\frac{|\xi_i|^2}{1+u_i^2}\right)^{-m-1-n-s}\!\!\!d\xi= \prod\limits_{i=1}^m(1+u_i^2)\,\frac{\pi^n}{\Gamma(n)}\,\int\limits_0^{+\infty}r^{n-1}(1+r)^{-m-1-n-s}\!dr,$$
where
\begin{equation}\label{defr}
r=r^{(m+1,n)}\left(z\right)=\sum\limits_{i=1}^n\frac{|\xi_i|^2}{1+u_i^2}\,.
\end{equation}

Recalling the Euler integral
\begin{equation*}
\int\limits_0^{+\infty}r^{n-1}(1+r)^{-m-1-n-s}\!dr=\frac{\Gamma(n)\!\cdot\!\Gamma(m+1+s)}{\Gamma(n+1+m+s)}\,,
\end{equation*}
we arrive at the desired conclusion.
Furthermore, introduce a map
$$\widetilde{\pi}^{m+1,n}_{\ m,\,n}\colon\Mat\bigl((m+1)\!\times\!n,\mathbb{C}\bigr)\to
\Mat\bigl(m\!\times\!n,\mathbb{C}\bigr)\times\mathbb{R}_+$$
by the formula
$$\widetilde{\pi}^{m+1,n}_{\ m,\,n}\left(z\right)=
\left(\pi^{m+1,n}_{\ m,\,n}\left(z\right),\,r^{(m+1,n)}\left(z\right)\right)\,,$$
where $r^{(m+1,n)}\left(z\right)$ is given by the formula \eqref{defr}.

Let $P^{(m,n,s)}$ be a probability measure on $\mathbb{R}_+$ given by the formula:
$$dP^{(m,n,s)}(r)=\frac{\Gamma(n+m+s)}{\Gamma(n)\!\cdot\!\Gamma(m+s)}\,r^{n-1}(1-r)^{-m-n-s}dr\,.$$

The measure $P^{(m,n,s)}$ is well-defined as soon as $m+s>0$.

 \begin{corollary}
For any $m,n\in\mathbb{N}$ and $s>-m-1$, we have
$$\left(\widetilde{\pi}^{m+1,n}_{\ m,\,n}\right)_*\mu^{(s)}_{m+1,n}= \mu^{(s)}_{m,n}\times P^{(m+1,n,s)}\,.$$
\end{corollary}

Indeed, this is precisely what was shown by our computation.

Removing a column is similar to removing a row:

$$\left(\pi^{m,n+1}_{\ m,\,n}\left(z\right)\right)^t=\pi^{m+1,n}_{\ m,\,n}\left(z^t\right).
$$
Write $\widetilde{r}^{(m,n+1)}\left(z\right)=r^{(n+1,m)}\left(z^t\right)$.
Introduce a map
$$\widetilde{\pi}^{m,n+1}_{\ m,\,n}\colon \Mat\bigl(m\!\times\!(n+1),\mathbb{C}\bigr)\to
\Mat\bigl(m\!\times\!n,\mathbb{C}\bigr)\times\mathbb{R}_+$$
by the formula
$$\widetilde{\pi}^{m,n+1}_{\ m,\,n}\left(z\right)=
\left(\pi^{m,n+1}_{\ m,\,n}\left(z\right),\,\widetilde{r}^{(m,n+1)}\left(z\right)\right)\,.$$


 \begin{corollary}
For any $m,n\in\mathbb{N}$ and $s>-m-1$, we have
$$\left(\widetilde{\pi}^{m,n+1}_{\ m,\,n}\right)_*\mu^{(s)}_{m,n+1}= \mu^{(s)}_{m,n}\times P^{(n+1,m,s)}\,.$$
\end{corollary}

Now take $n$ such that $n+s>0$ and introduce a map
$$\widetilde{\pi}_n\colon \Mat\bigl(\mathbb{N}\!\times\!\mathbb{N},\mathbb{C}\bigr)\to
\Mat\bigl(n\!\times\!n,\mathbb{C}\bigr)\times\mathbb{R}_+^{\infty}$$
by the formula
$$\widetilde{\pi}_n\left(z\right)=
\left(\pi^{\infty,\infty}_{\ n,\,n}\left(z\right),\,r^{(n+1,n)},\,\widetilde{r}^{(n+1,n+1)},\, r^{(n+2,n+1)},\,\widetilde{r}^{(n+2,n+2)},\dots\right)\,.$$

Recalling the definition of the Pickrell measures $\mu^{(s)}$ (see subsection~\ref{sec:Pickrell_measures}), we can now reformulate the result of our computations as follows:

\begin{proposition}
If $n+s>0$, then we have
\begin{equation}\label{inf-prod-s}
\left(\widetilde{\pi}_n\right)_*\mu^{(s)}=\mu^{(s)}_{m,n}\times\prod\limits_{l=0}^{\infty} \left(P^{(n+l+1,n+l,s)}\!\times\!P^{(n+l+1,n+l+1,s)}\right)\,.
\end{equation}
 \end{proposition}

\end{document}